\date{\today}
\documentclass[10pt,reqno]{amsart}
\usepackage{amsthm,amsmath,verbatim,lscape,color,tensor,url}
\usepackage{stmaryrd}
\usepackage{MnSymbol}

\usepackage[margin=1.5in]{geometry}
\usepackage[enableskew]{youngtab}
\usepackage[all]{xy}
\usepackage{enumerate}
\usepackage{paralist}
\usepackage{mathrsfs}
\usepackage{caption}

\usepackage{tikz}
\usetikzlibrary{matrix,arrows,decorations.pathmorphing,
  cd,patterns,calc}

\usepackage{xr-hyper}
\usepackage{hyperref}
\usepackage[shortalphabetic]{amsrefs}

\usepackage{graphicx}

\hypersetup{
  bookmarksnumbered=true,%
  colorlinks=true,%
  linkcolor=blue,%
  citecolor=blue,%
  filecolor=blue,%
  menucolor=blue,%
  urlcolor=blue,%
  pdfnewwindow=true,%
  pdfstartview=FitBH}
\pdfcatalog{/PageLayout /SinglePage}

\usepackage[nameinlink,capitalise,noabbrev]{cleveref}
\newtheoremstyle{named}%
    {}{}{\itshape}{}{\bfseries}{.}{.5em}{\thmnote{#3}}
\theoremstyle{named}

\crefname{equation}{}{}
\crefname{introcor}{Corollary}{Corollary}

\Crefname{figure}{Figure}{Figures}

\theoremstyle{plain}
\newtheorem{thm}{Theorem}[section]
\newtheorem{theorem}[thm]{Theorem}
\newtheorem*{thm*}{Theorem}

\newtheorem{corollary}[thm]{Corollary}
\newtheorem{prop}[thm]{Proposition}
\newtheorem{proposition}[thm]{Proposition}
\newtheorem{lemma}[thm]{Lemma}

\theoremstyle{definition}

\newtheorem{example}[thm]{Example}

\newtheorem{remark}[thm]{Remark}

\newtheorem{notn}[thm]{Notation}

\newenvironment{pf}{\begin{proof}}{\end{proof}}

\DeclareMathOperator{\Ho}{Ho}

\DeclareMathOperator{\CAlg}{CAlg}

\DeclareMathOperator{\hoco}{hocolim}

\DeclareMathOperator{\Id}{Id}

\DeclareMathOperator{\fib}{fib}

\DeclareMathOperator{\Mod}{\mathrm{Mod}}

\DeclareMathOperator{\Sp}{Sp}

\DeclareMathOperator{\Cl}{Cl}

\DeclareMathOperator{\hocolim}{hocolim}

\newcommand{\C}{\mathbb{C}}

\newcommand{\Z}{\mathbb{Z}}
\newcommand{\F}{\mathbb{F}}
\newcommand{\Q}{\mathbb{Q}}
\newcommand{\R}{\mathbb{R}}

\newcommand{\cF}{\mathcal{F}}

\newcommand{\cP}{\mathcal{P}}

\newcommand{\lra}[1]{\overset{#1}{\longrightarrow}}

\makeatletter
\let\c@equation\c@thm
\makeatother

\numberwithin{equation}{section}

\DeclareMathOperator{\Set}{Set}
\DeclareMathOperator{\Gal}{Gal}

\DeclareMathOperator{\coker}{coker}

\newcommand{\ul}{\underline}
\newcommand{\rtarr}{\longrightarrow}
\newcommand{\xrtarr}{\xrightarrow}
\newcommand{\smsh}{\wedge}

\newcommand{\id}{\mathrm{id}}
\newcommand{\iso}{\cong}
\newcommand{\mf}[1]{\ul{#1}}
\newcommand{\mpi}{\mf{\pi}}

\newcommand{\longto}{\longrightarrow}%
\newcommand{\onto}{\twoheadrightarrow}%
\usepackage{harpoon}

\usepackage[final]{microtype}

\DeclareMathOperator{\Mack}{Mack}

\makeatletter
\newcommand*{\da@rightarrow}{\mathchar"0\hexnumber@\symAMSa 4B }
\newcommand*{\da@leftarrow}{\mathchar"0\hexnumber@\symAMSa 4C }
\newcommand*{\xdashrightarrow}[2][]{%
  \mathrel{%
    \mathpalette{\da@xarrow{#1}{#2}{}\da@rightarrow{\,}{}}{}%
  }%
}
\newcommand*{\da@xarrow}[7]{%
  \sbox0{$\ifx#7\scriptstyle\scriptscriptstyle\else\scriptstyle\fi#5#1#6\m@th$}%
  \sbox2{$\ifx#7\scriptstyle\scriptscriptstyle\else\scriptstyle\fi#5#2#6\m@th$}%
  \sbox4{$#7\dabar@\m@th$}%
  \dimen@=\wd0 %
  \ifdim\wd2 >\dimen@
    \dimen@=\wd2 %
  \fi
  \count@=2 %
  \def\da@bars{\dabar@\dabar@}%
  \@whiledim\count@\wd4<\dimen@\do{%
    \advance\count@\@ne
    \expandafter\def\expandafter\da@bars\expandafter{%
      \da@bars
      \dabar@ 
    }%
  }%
  \mathrel{#3}%
  \mathrel{%
    \mathop{\da@bars}\limits
    \ifx\\#1\\%
    \else
      _{\copy0}%
    \fi
    \ifx\\#2\\%
    \else
      ^{\copy2}%
    \fi
  }%
  \mathrel{#4}%
}
\makeatother

\newcommand{\KUG}{KU_G}

\title{On the $KU_G$-local equivariant sphere}
\author[P. J. Bonventre]{Peter {J}. Bonventre}
\email{peter.bonventre@georgetown.edu}
\author[B. J. Guillou]{Bertrand {J}. Guillou}
\email{bertguillou@uky.edu}
\author[N. J. Stapleton]{Nathaniel {J}. Stapleton}
\email{nat.j.stapleton@uky.edu}
\date{\today}
\thanks{
Guillou was supported by NSF grant DMS-2003204.
Stapleton was supported by NSF grant DMS-1906236 and a Sloan Fellowship.
}

\begin{document}

\maketitle

\begin{abstract}
Equivariant complex $K$-theory and the equivariant sphere spectrum are two of the most fundamental equivariant spectra. For an odd $p$-group, we calculate the zeroth homotopy Green functor of the localization of the equivariant sphere spectrum with respect to equivariant complex $K$-theory. Further, we calculate the zeroth homotopy Tambara functor structure in the case of odd cyclic $p$-groups.
\end{abstract}

\tableofcontents

\section{Introduction}

In the '70s, Adams--Baird (unpublished) and Ravenel \cite{Rav}, see also \cite{Bousfield}, calculated the homotopy groups of the $KU$-localization of the sphere spectrum. In particular, they found that $\pi_0 L_{KU}S \cong \Z \oplus \F_2$. In fact, the natural homotopy commutative ring structure on $L_{KU}S$ endows $\pi_0L_{KU}S$ with the ring structure $\mathbb Z[x]/(2x,x^2)$. Both complex $K$-theory $KU$ and the sphere spectrum $S$ admit natural equivariant refinements. Let $q$ be an odd prime. The goal of this paper is to calculate the Mackey functor $\mpi_0$ of the localization of the equivariant sphere spectrum with respect to equivariant complex $K$-theory when the group of equivariance is a $q$-group. 

Fix a $q$-group $G$. Let $KU_G$ be genuine equivariant complex $K$-theory and let $S_G$ be the genuine equivariant sphere spectrum. We will denote the zeroth homotopy Mackey or Green functor by $\mpi_0$. 
Both $\mpi_0 KU_G$ and $\mpi_0 S_G$ admit concrete descriptions --- $\mpi_0 KU_G \cong \mf{RU}$, the complex representation ring Green functor, and $\mpi_0 S_G \cong \mf{A}$, the Burnside ring Green functor.

Our goal is to understand $\mpi_0 L_{KU_G} S_G$, the equivariant generalization of the result of Adams--Baird and Ravenel mentioned above. As $S_G$ is an $E_\infty$-ring in $G$-spectra, \cite{Hill19}*{Corollary~3.12} implies $L_{KU_G} S_G$ is again $E_\infty$, and hence $\mpi_0 L_{KU_G} S_G$ is a Green functor. 
In fact, we show in \cref{LKUGSH_GEinf} that $L_{KU_G} S_G$ is a $G$-$E_\infty$ ring.
This implies that $\mpi_0 L_{KU_G} S_G$ is furthermore a Tambara functor, and we determine this structure in the case that $G$ is an odd cyclic $q$-group.

Let $\mf{J} \subset \mf{A}$ be the Mackey ideal with $\mf{J}(G/H) \subset \mf{A}(G/H)$ generated by virtual $H$-sets $X$ such that $|X^h| = 0$ for all $h \in H$. 
We show the following:

\begin{theorem}
\label{thm:main} 
Let $G$ be an odd $q$-group. Then there is an isomorphism of Green functors 
\[
\underline{\pi}_0 L_{KU_G} S_G \cong
(\underline{A}/\underline{J}) \otimes \pi_0(L_{KU}S)
\cong
(\underline{A}/\underline{J}) [x]/(2x,x^2).
\]
\end{theorem}

Our work on this result was motivated by two things. First was our desire to understand the genuine equivariant analogue of a question of Ravenel's \cite{ProbSess} about the kernel of the canonical map from the Burnside ring to the $K(n)$-local cohomotopy of $BG$ when $n=1$. \cref{thm:main} is certainly the kind of answer that Ravenel would have expected. See also \cite{Szymik}*{Section~4.2} for Ravenel's question when $n=1$. Second was our desire to understand how to calculate with localizations in genuine equivariant stable homotopy theory. We learned that the geometric fixed point functors are the most powerful tools in the toolkit. From this perspective, we view \cref{thm:main} as a first nontrivial exercise to solve.

To prove \cref{thm:main}, we follow the standard strategy for calculating the homotopy groups of the $KU$-local sphere, adding in some applications of the geometric fixed point functors when needed. That is, we use the arithmetic fracture square \cref{eq:fracsq} in order to work locally at a prime $p$. The calculation looks different when $p$ is equal to $q$ in comparison to when $p$ is different from $q$.

We show that for $\ell$ coprime to $q$ and furthermore primitive mod $q^k$ for all $k>0$, there is a fiber sequence of equivariant spectra
\[
L_{KU_G/q} S_G \to (KU_G)^{\wedge}_q \overset{\psi^\ell-1}{\longrightarrow} (KU_G)^{\wedge}_q, 
\]
and we use this to calculate $\mpi_0L_{KU_G/q} S_G$. This requires that we show that $\psi^\ell$ is stable after inverting $\ell$ and also uses the fact that $G$ is a $q$-group to describe the kernel of $\mpi_0 (\psi^\ell-1)$ in terms of the Burnside ring. To see that $\psi^\ell$ is stable after inverting $\ell$, we make use of the Atiyah--Segal character map and formulas for the Adams operations obtained by the third author with Barthel and Berwick-Evans.

When $p \neq q$, we calculate $\mpi_0L_{KU_G/p} S_G$ using the product decomposition of the category of equivariant spectra localized away from the order of the group. 
In this case, the collection of geometric fixed point functors can be used to produce an equivalence between the category of equivariant spectra (localized away from $|G|$) and the product over conjugacy classes of subgroups $H \subseteq G$ of the categories of $p$-local Borel-equivariant $W(H)$-spectra, where $W(H)$ is the Weyl group of $H$ in $G$. 
There is also an algebraic incarnation of this equivalence. 
The key result underlying both is that, after inverting the order of the group, the Burnside ring factors as a product of copies of $\Z[1/{|G|}]$. This leads to a corresponding decomposition of the 
 category of $p$-local $G$-Mackey functors as a product of simpler algebraic categories. We give an explicit formula for the inverse to these equivalences. Making use of the facts that geometric fixed points send localizations to localizations and that $\Phi^H KU_G$ is trivial unless $H \subseteq G$ is cyclic, it is reasonably straight forward to find $\mpi_0 L_{KU_G/p} S_G$ in this case.

 \subsection*{Acknowledgements} 
It is a pleasure to thank Tomer Schlank for his invaluable input. He caught an error in our first version of the argument and suggested the work-around. From the beginning of the project he encouraged us to exploit the geometric fixed point functors. We also thank John Greenlees for a helpful discussion concerning the geometric fixed points of $KU_G$.  We are indebted to Mike Geline for making us aware of Schilling's theorem. Further, we thank William Balderrama for several specific suggestions. We also thank Anna Marie Bohmann, Davis Deaton, and Mike Hill for helpful discussions.

\subsection{Organization}

We begin with a quick review of background material in \cref{sec:review}.
In \cref{sec:geometric} we calculate the geometric fixed points of $KU_G$.
We analyze the behavior of the Atiyah-Segal character map on a certain equivariant Bott class in \cref{sec:Bott}
and use this in \cref{sec:stable} to show that the Adams operation $\psi^\ell$ lifts to a map of equivariant spectra.
This leads to the above fiber sequence, which we use to calculate $\mpi_0L_{KU_G/q} S_G$ in \cref{sec:fiber}.
We review the $p$-local splitting of $G$-spectra and its algebraic analogue in \cref{sec:splitting} and use this in \cref{sec:pneqq} to describe $L_{KU_G/p}S_G$ when $p\neq q$.
In \cref{sec:fracturesquare} we synthesize these calculations in a fracture square to prove our main result, \cref{thm:main}.
In the final \cref{sec:Hillforthewin}, we show that $L_{KU_G} S_G$ inherits a $G$-$E_\infty$ structure and calculate the Tambara functor structure on $\mpi_0L_{KU_G/q} S_G$, when $G$ is an odd cyclic $q$-group.

\section{Preliminaries}
\label{sec:review}

For the duration of the paper, we fix a finite group $G$. At times, we will assume it is further an (odd) $q$-group. In this section we describe the algebraic and topological objects that will play a role in the rest of the paper.
\subsection{Algebra}
We will make use of several commutative rings associated to $G$:
\begin{itemize}
    \item Let $A(G)$ be the Burnside ring. This is the Grothendieck group of isomorphism classes of finite left $G$-sets under disjoint union. The product is induced by the product of left $G$-sets.
    \item Let $R\Q(G)$ (resp. $RO(G)$, $RU(G)$) be the rational (resp. real, complex) representation ring. This is the Grothendieck group of isomorphism classes of finite-dimensional rational (resp. real, complex) $G$-representations under direct sum, with product induced by the tensor product of $G$-representations.
    \item For $\Q \subseteq R \subseteq \C$, let $\Cl(G,R)$ be the ring of $R$-valued class functions on $G$.
    This is the ring of $R$-valued functions on the set of conjugacy classes of $G$.
    \item Let $\chi \colon RU(G) \to \Cl(G,\C)$ be the character map. This map is injective and thus $RU(G)$ may be viewed as a subring of $\Cl(G,\C)$. 
    \item Let $R\Q_{\chi}(G)$ be the subring of $RU(G)$ 
    consisting of virtual representations for which the character takes rational values. In other words, $R\Q_{\chi}(G) = RU(G) \cap \Cl(G,\Q)$, where the intersection takes place in $\Cl(G,\C)$.
\end{itemize}

There are canonical ring maps
\begin{equation} \label{eq:all}
A(G) \to R\Q(G) \to RO(G) \to RU(G) \to \Cl(G,\C),
\end{equation}
none of which are necessarily isomorphisms. The first map is induced by the operation that sends a finite $G$-set to the free rational vector space on the underlying set, the second map is induced by base change from $\Q$ to $\R$, the third map is induced by base change from $\R$ to $\C$, and the fourth map is the character map $\chi$. The ring $R\Q_{\chi}(G)$ sits in between $R\Q(G)$ and $RU(G)$.

Recall that a Green functor is a Mackey functor that takes values in commutative rings and for which the restriction maps are ring maps and the transfer maps satisfy Frobenius reciprocity. 
Equivalently, a Green functor is a commutative monoid in the symmetric monoidal category of Mackey functors
\cite{Lewis}.
Each of the constructions above extends to a Green functor. We will denote the associated Green functors with an underline. For example, $\mf{A}$ is the $G$-Green functor defined by $\mf{A}(G/H) = A(H)$. As all of the maps in \ref{eq:all} are compatible with restriction and transfer maps, we have maps of Green functors
\begin{equation} \label{eq:mackeyall}
\mf{A} \to \mf{R\Q} \to \mf{RO} \to \mf{RU} \to \mf{\Cl}.
\end{equation}

In fact, each of these constructions, as well as $\mf A / \mf J$, extend to (maps of) Tambara functors, an even richer algebraic structure. However, we will not make use of this observation.

\subsection{Review of $G$-spectra}
We will work throughout in the category $Sp^G$ of (genuine) $G$-equivariant spectra.
We will make use of several equivariant ring spectra:
\begin{itemize}
    \item Let $S_{G}$ be the equivariant sphere spectrum.
    \item For a $G$-Mackey functor $\ul{M}$, let $H_G \ul{M}$ be the equivariant Eilenberg-Mac~Lane spectrum
    \item Let $\KUG$ be the equivariant complex $K$-theory spectrum.
\end{itemize}

We will use the same notation for a pointed $G$-space and its  suspension $G$-spectrum. A cofiber sequence of pointed $G$-spaces then gives rise to a cofiber sequence of $G$-spectra via the suspension $G$-spectrum functor.

Given $G$-spectra $E$ and $X$, we will write $L_{E}X$ for the Bousfield localization of $X$ with respect to $E$. This construction has been studied previously in \cite{Carrick}. If $X$ is an $E_{\infty}$-ring in equivariant spectra, then $L_{E}X$ is an $E_{\infty}$-ring in equivariant spectra. Further, we will write
\[
X^{\wedge}_p = L_{M(p)_G}X
\]
where $M(p)_G = S_G/p$ is the mod $p$ Moore spectrum,
and 
\[
\Q \otimes X = L_{H_G\ul{\Q}} X \simeq H_G\ul{\Q} \smsh X.
\]
These fit together into the arithmetic fracture square
\begin{equation}
    \label{eq:fracsq}
    \xymatrix{X \ar[r] \ar[d] & \prod_{p} X^{\wedge}_p \ar[d] \\ \Q \otimes X \ar[r] & \Q \otimes (\prod_p X^{\wedge}_p),}  
\end{equation}
which is a homotopy pullback of equivariant spectra.
If $X$ has the structure of an $E_\infty$-ring $G$-spectrum, then this is a homotopy pullback of $E_\infty$-ring $G$-spectra.

For a $G$-spectrum $X$, the Mackey functor $\mpi_n(X)$ has values 
\[ \mpi_n(X)(G/H) = \pi_n^H (X) = \pi_n(X^H),\]
where $X^H$ is the fixed-point spectrum, as in \cref{sec:ReviewFix} below.
Some of the Green functor-valued homotopy groups of some of the equivariant spectra above are well-known:
\begin{itemize}
    \item $\mpi_0 S_G \cong \mf{A}$.
    \item $\mpi_* \KUG \cong \mf{RU}[\beta, \beta^{-1}]$, where $\beta$ is in degree $2$.
\end{itemize}

We will also make use of the category $Sp^{hG}$ of Borel G-equivariant spectra. This is the localization of $Sp^G$ at the set of underlying equivalences. The localization functor $Sp^G \rtarr Sp^{hG}$ has both a left and a right adjoint. The left adjoint is given by $X \mapsto EG_+ \smsh X$, while the right adjoint is $X \mapsto F(EG_+,X)$.

\subsection{Fixed points and geometric fixed points}
\label{sec:ReviewFix}

For any subgroup $H \leq G$, there is a restriction-induction adjunction
\[
\begin{tikzcd}
    Sp^{G} \ar[r,"\downarrow^G_{H}",shift left=0.5ex] & Sp^{H} \ar[l,"\uparrow_{H}^G",shift left=0.5ex]
\end{tikzcd}
\]
between the category of $G$-spectra and $H$-spectra.
According to the Wirthmuller isomorphism, restriction is both left and right adjoint to induction.

Suppose that $N \unlhd G$ is a normal subgroup. Then there is an adjoint pair
\begin{equation}
\label{InflationAdjunction}
\begin{tikzcd}
    Sp^{G/N} \ar[r,"\mathrm{inf}_{G/N}^G",shift left=0.5ex] & Sp^G, \ar[l,"(-)^N",shift left=0.5ex]
\end{tikzcd}
\end{equation}
where $\inf_{G/N}^G$ is inflation and $(-)^N$ is the categorical $N$-fixed points functor. We will also heavily employ the geometric fixed points functor, which fits into an adjunction as
\begin{equation}
\label{GeometricAdjunction}
\begin{tikzcd}
    Sp^{G} \ar[r,"\Phi^N",shift left=0.5ex] & Sp^{G/N}. \ar[l,"\phi_N^*",shift left=0.5ex]
\end{tikzcd}
\end{equation}
Denote by $\cF[N]$ the family of subgroups of $G$ which do not contain $N$, and let 
$\widetilde{E\cF[N]}$  be the cofiber of $E\cF[N]_+\rtarr S_G$.
Then the geometric fixed points functor is 
\[\Phi^N(X) = ( \widetilde{E\cF[N]} \smsh X)^N,\]
while the geometric inflation functor is
\[
\phi_N^*(Y) = \widetilde{E\cF[N]} \smsh \mathrm{inf}_{G/N}^G Y.
\]
In the case $N=G$, then $\cF[G]$ is the family of proper subgroups, which we will write as $\mathcal{P}_G$. 
Both fixed point functors can be extended to the case of a not-necessarily-normal subgroup $H\leq G$ by composing with the restriction-induction adjunction
\[
\begin{tikzcd}
    Sp^{G} \ar[r,"\downarrow^G_{N_GH}",shift left=0.5ex] & Sp^{N_GH}, \ar[l,"\uparrow_{N_GH}^G",shift left=0.5ex]
\end{tikzcd}
\]
where $N_GH \leq G$ is the normalizer of $H$ in $G$.

\section{Geometric fixed points of $\KUG$} \label{sec:geometric}

In this section, we will compute the geometric fixed points of $KU_G$ at $q$-subgroups where $q$ is a prime,
following \cite{Gr}. 

\begin{notn}
As usual, we will use $\rho = \rho_G$ to denote the complex regular representation of $G$.
\end{notn}

The categorical fixed points of $\KUG$ were calculated by Segal.

\begin{prop}\cite{Seg}*{Proposition~2.2} \label{CatFPKU}
There is an equivalence of homotopy commutative ring spectra $(\KUG)^H \simeq KU \otimes RU(H)$.
\end{prop}

In other words, we have that the categorical fixed points are a free $KU$-module, of rank equal to the number of conjugacy classes in $G$.
\medskip

We begin by computing geometric fixed points with respect to the cyclic subgroups $C_{q^k} \leq G$. 
We will see below in \cref{PhiNonCycKUzero} that if $H\leq G$ contains a non-cyclic $q$-group, then $\Phi^H(\KUG)\simeq \ast$, so that cyclic groups are the only case of interest.
We denote by $\overline{RU}(C_{q^k})$ the quotient 
\begin{equation}
\label{defnRUbar}
\overline{RU}(C_{q^k}) = RU(C_{q^k})/\rho(k-1),
\end{equation}
where $\rho(k-1)$ is the pullback of $\rho_{C_q} \in RU(C_q)$ along the quotient map $C_{q^k} \rtarr C_q$. 

Fix an isomorphism $RU(C_{q^k}) \iso \Z[x]/(x^{q^k}-1)$, where $x$ denotes a nontrivial 1-dimensional irreducible representation of $C_{q^k}$. In this ring, the regular representation is given by
\[
\rho_{C_{q^k}} = x^{{q^k}-1}+ \dots + x + 1.
\]  
We can identify $\rho(k-1)$ under this isomorphism as 
\[
\rho(k-1) = x^{(q-1)q^{k-1}}+\dots+x^{q^{k-1}}+1,
\]
and the defining relation for $RU(C_{q^k})$ splits as the product 
 \begin{equation} \label{regularfactorization}
 x^{q^k}-1 = (x^{q^{k-1}}-1) \cdot \rho(k-1).
 \end{equation}
We see that $\overline{RU}(C_{q^k})$ is a free abelian group of rank $(q-1)q^{k-1}$.

\begin{proposition}
\label{CyclicGeoFP}
Let $C_{q^k} \subseteq G$ be a cyclic $q$-subgroup of $G$. There is an equivalence of $KU$-modules 
\[
\Phi^{C_{q^k}} \KUG \simeq KU \otimes \overline{RU}(C_{q^k}) [\frac1q].
\]
\end{proposition}

\begin{proof}
As we are only considering the underlying spectrum of the geometric fixed points, as opposed to the more equivariantly sophisticated variant from \eqref{GeometricAdjunction}, we may without loss of generality suppose that $G=C_{q^k}$.
Recall that one model for the space $\widetilde{E\mathcal{P}_G}$ of \cref{sec:ReviewFix} is 
\[
\widetilde{E\mathcal{P}_G} = S^{\infty V} = \hocolim_j S^{jV},
\]
 where $V$ is a real $G$-representation such that $V^G=0$ and $V^H\neq 0$ for all proper subgroups $H$. 
For example, we may take $V=\overline{\rho}_\R = \rho_\R-1$, the reduced real regular representation.
It follows that for any $G$-spectrum $X$, the geometric fixed points can be computed as 
\[ 
\Phi^G X = ( S^{\infty V} \smsh X)^G \simeq \hoco_j (S^{jV} \smsh X)^G, 
\]
The maps in the colimit are given by multiplication by the Euler class $e_V\in \pi_{-V}(S_G)$ of $V$ on $\pi_{\star} X$.

Now, if $V$ is the underlying $2n$-dimensional real representation of an $n$-dimensional complex representation, then equivariant Bott periodicity (see \cite{At} or \cite{Alaska}*{Section XIV.4}) gives a canonical equivalence of equivariant spectra $\Sigma^V KU_{C_{q^k}} \simeq \Sigma^{2n} KU_{C_{q^k}}$.
If $q$ is odd, then $V=\overline{\rho}_\R$ underlies a complex representation of dimension $\frac{q^k-1}2$. 
On the other hand,  in the case of $C_2$, $\overline{\rho}_\R = \sigma_\R$ is the 1-dimensional sign representation, but $2\sigma_\R$ underlies the 1-dimensional complex sign representation $\sigma_\C$, and so we take $V=2\sigma_\R$ in this case. 
Similarly, for $C_{2^k}$, we take $V=2\overline{\rho}_\R$.
In either of these cases, the Euler class may be identified with a $\Z$-graded class, and the geometric fixed points may be rewritten as
\[ \Phi^{C_{q^k}} KU_{C_{q^k}} \simeq \hoco_j  \Sigma^{2jn} (KU_{C_{q^k}})^{C_{q^k}}.\]
According to \cref{CatFPKU} and the 2-fold Bott periodicity of $KU$, this is equivalent to
\[ \Phi^{C_{q^k}} KU_{C_{q^k}} \simeq \hoco_j KU \otimes RU(C_{q^k}),\]
where the maps in the colimit are multiplication by the Euler class of $V$, thought of as a class in degree 0 via Bott periodicity.
In other words, we are inverting the image of the Euler class in $KU\otimes RU(C_{q^k})$.

Under the fixed isomorphism $RU(C_{q^k}) \iso \Z[x]/(x^{q^k}-1)$ from above,
the reduced complex regular representation can be identified as
\[
\overline{\rho} = \overline{\rho}_{C_{q^k}} = x^{q^k-1} + \dots + x.
\]
Then 
\begin{equation}
\label{EulerRho}
e(\overline{\rho}) = e\left( \sum_{i=1}^{q^k-1} x^i \right) = \prod_{i=1}^{q^k-1}( x^i-1).
\end{equation}
The maps in the colimit computing the geometric fixed points are given by multiplying by this Euler class, so it remains to understand the effect of inverting this class in $RU(C_{q^k})$. We carry this out in \cref{LocRUEuler} and \cref{LocRUx-1}.
\end{proof}

Let $V$ and $W$ be complex representations with $KU$-theory Euler classes $e(V)$ and $e(W)$ in $RU(G)$. Recall that $e(V\oplus W) = e(V) \cdot e(W)$, and further, if $V$ is 1-dimensional, then $e(V) = V-1$.

\begin{lemma}\label{LocRUEuler}
The localization $RU(C_{q^k}) [ \frac1{e(\overline{\rho})}]$ is isomorphic to 
$RU(C_{q^k}) [ \frac1{x^{q^{k-1}}-1}]$.
\end{lemma}

\begin{proof}

 Since $x^{q^{k-1}}-1$ is a factor of $e(\overline{\rho})$ according to \cref{EulerRho}, it is clear that inverting the Euler class also inverts $x^{q^{k-1}}-1$. 

Conversely, $\eqref{regularfactorization}$ implies that there is an isomorphism
\begin{equation}\label{RUCp1overxminus1}
 RU(C_{q^k})\left[ \frac1{x^{q^{k-1}}-1}\right] \iso \Z[x]/\rho(k-1) \left[ \frac1{x^{q^{k-1}}-1}\right].
\end{equation} 
For any $j<k-1$, the class $x^{q^j}-1$ divides $x^{q^{k-1}}-1$ and therefore becomes invertible after inverting $x^{q^{k-1}}-1$. It remains to consider $x^i-1$, where $i$ is prime to $q$.
Since $x^i-1 = (x-1)(x^{i-1}+\dots+1)$, and $x-1$ has already been inverted,
it suffices
by \cref{RUCp1overxminus1} to show that $\rho_i = x^{i-1}+\dots+1$ is invertible in 
$\Z[x]/\rho(k-1)$ for $i$ prime to $q$. This follows from the fact that if $i$ is prime to $q$ then $\rho_i$ and $\rho(k-1)$ do not share any common roots (over $\C$).
\end{proof}

\begin{lemma}\label{LocRUx-1}
 The localization $RU(C_{q^k}) [ \frac1{x^{q^{k-1}}-1}]$ is isomorphic to 
 \[ 
 \overline{RU}(C_{q^k}) \left[ \frac1q\right] 
 \iso \Z\left[x,\frac1q\right]/(1+x^{q^{k-1}}+\dots+x^{(q-1)q^{k-1}}),\]
 where $\overline{RU}(C_{q^k})$ is as in \cref{defnRUbar}.
\end{lemma}

\begin{proof}
According to \cref{RUCp1overxminus1}, it suffices to show that in $\overline{RU}(C_{q^k}) \iso \Z[x]/\rho(k-1)$, inverting $x^{q^{k-1}}-1$ agrees with inverting $q$. For simplicity, we write $y=x^{q^{k-1}}$ in the rest of this argument.

On the one hand, $(y-1)^q \equiv y^q-1 \pmod{q}$. Since $y^q-1=0$ in $RU(C_{q^k})$, we conclude that $(y-1)^q$ is divisible by $q$ in $RU(C_{q^k})$ (and therefore also in the quotient $\overline{RU}(C_{q^k})$). It follows that inverting $y-1$ also inverts $q$. 

On the other hand, we can check directly that
\[ (1-y) \cdot  (y^{q-2} + 2 y^{q-3} + 3 y^{q-4} + ... + (q-2) y + (q-1) ) = -\rho(k-1)+ q = q\]
in $\overline{RU}(C_{q^k}) \cong \Z[x]/\rho(k-1)$.
Therefore inverting the integer $q$ also inverts $y-1$ in the ring $\Z[x]/\rho(k-1)$.
\end{proof}

We now study the geometric fixed points with respect to a non-cyclic $q$-subgroup.

\begin{proposition}
\label{PhiNonCycKUzero}
Suppose that $H\leq G$ is a non-cyclic $q$-subgroup. Then $\Phi^H \KUG \simeq \ast$.
\end{proposition} 

\begin{proof}
By restriction, it suffices to consider the case $H=G$.
If $G$ is not cyclic, then it admits a surjection to $C_q\times C_q$. This induces a ring map $\Phi^{C_q\times C_q} KU_{C_q \times C_q} \rtarr \Phi^G \KUG$ as follows.

More generally, given a surjection $q_N \colon G \rtarr G/N$, there is a canonical map of $G$-spaces $q_N^* \widetilde{E\cP}_{G/N} \rtarr \widetilde{E\cP}_G$. 
Moreover, as $KU$ is a global ring spectrum, it comes equipped with a  map of ring $G$-spectra 
$\mathrm{inf}_{G/N}^G  KU_{G/N} \rtarr KU_G$ (see \cite{LMS}*{II.8.5}).
Adjoint to this is a
map of ring $G/N$-spectra $\xi\colon KU_{G/N} \rtarr (KU_G)^N$.
Then the desired map on geometric fixed points is
\[ 
    \begin{tikzcd}
        \Phi^{G/N} KU_{G/N} = \left( \widetilde{EP}_{G/N} \wedge KU_{G/N} \right)^{G/N} 
        \ar[d,"\id \wedge \xi"] & \left(\widetilde{EP}_G \wedge KU_G\right)^G = \Phi^G KU_G,
          \\
         \left(  \widetilde{EP}_{G/N} \wedge (KU_{G})^N \right)^{G/N} \ar[r,"\sim"]
         & \left( q^*\widetilde{EP}_{G/N} \wedge KU_G \right)^G \ar[u]
    \end{tikzcd}
\]
where the equivalence on the second row is the projection formula (see \cite{BalmerSanders}*{2.(C)} or \cite{HKBPO}*{Lemma~2.13}).

It now remains to show that $\Phi^{C_q\times C_q}KU_{C_q\times C_q}\simeq \ast$. The geometric fixed points are computed by inverting the Euler classes of nontrivial irreducible representations of $C_q \times C_q$ in $KU\otimes RU(C_q\times C_q)$. Now 
\[ 
RU(C_q\times C_q) \iso \Z[x,y]/(x^q-1,y^q-1).
\]
According to \cref{LocRUx-1}, inverting the Euler class $x-1$ gives
\[
RU(C_q\times C_q)\left[\frac1{x-1}\right] \iso \Z\left[x,y,\frac1q\right]/(x^{q-1}+\dots+x+1,y^q-1).
\]
In this localization, $x$ is a $q$th root of unity, so that
\[ 
y^q-1 = \prod_{i=0}^{q-1} (y-x^i) =
x^{\frac{q(q-1)}2} \prod_{i=0}^{q-1} (y x^{q-i} -1) =
\prod_{i=0}^{q-1} e(y x^{q-i}).
\]
Thus inverting the Euler classes $e(yx^{q-i})$ will invert $y^q-1$, which is zero in $RU(C_q\times C_q)$. It follows that the localization is zero.
\end{proof}

\begin{remark}
In the proposition above, the map of ring spectra $\Phi^{C_q\times C_q} KU_{C_q \times C_q} \rtarr \Phi^G \KUG$ is strictly more than we need to prove the result. It suffices to know that the Euler classes inverted in $RU(C_q \times C_q)$ in the formula for $\pi_0(\Phi^{C_q\times C_q} KU_{C_q \times C_q})$ are inverted in $RU(G)$ in the formula for $\pi_0(\Phi^G \KUG)$. Given this, it follows that $\pi_0(\Phi^G \KUG) = 0$, which implies that $\Phi^G \KUG \simeq *$ as it is a commutative ring spectrum.
\end{remark}


\section{The character of the equivariant Bott classes}
\label{sec:Bott}

Let $V$ be a finite-dimensional complex representation of the finite group $G$.
The Thom isomorphism in equivariant complex $K$-theory is a canonical isomorphism of $RU(G)$-modules
\[
\KUG^0(*) \cong \widetilde{KU}_G^0(S^{V}),
\]
where $S^{V}$ is the representation sphere associated to $V$. This isomorphism is given by multiplication by the equivariant Bott class $\beta^{V}$ (see \cite{Seg}*{Section 3} and \cite{Alaska}*{Section XIV.4}).
Thus $\widetilde{KU}_G^0(S^V)$ is a free module of rank one over $RU(G)$ on the class $\beta^{V}$:
\[
\widetilde{KU}_G^0(S^{V}) \cong RU(G)\{\beta^{V}\}.
\]
This algebraic statement is a consequence of the topological statement of Bott periodicity in the proof of \cref{CyclicGeoFP}. Given two finite dimensional complex $G$-representations $V$ and $W$, the canonical isomorphism of $RU(G)$-modules
\[
\widetilde{KU}_G^0(S^{V}) \otimes_{RU(G)} \widetilde{KU}_G^0(S^{W}) \cong \widetilde{KU}_G^0(S^{V \oplus W})
\]
sends $\beta^{V} \otimes \beta^{W}$ to $\beta^{V \oplus W}$.

Let $\rho_G$ be the complex regular representation of $G$ and let
\[
\chi \colon \widetilde{KU}_G^0(S^{\rho_G}) \rtarr \prod_{[g]} \widetilde{H}^0((S^{\rho_G})^g, \C[\beta,\beta^{-1}])
\]
be the Atiyah--Segal character map \cite{AtSe} applied to the finite $G$-CW complex $S^{\rho_G}$.

\begin{prop}
\label{CharacterBottRhoG}
The image of $\beta^{\rho_G}$ under the Atiyah--Segal character map $\chi$ is the ``class function" $\chi(\beta^{\rho_G})$ sending
\[
[g] \mapsto (|g|\beta)^{|G|/|g|}.
\]
\end{prop}

\begin{proof}
Fix a conjugacy class $[g] \subseteq G$ and $g \in [g]$. We will describe the part of the character map
\[
\chi \colon \widetilde{KU}_G^0(S^{\rho_G}) \to  \widetilde{H}^0((S^{\rho_G})^g, \C[\beta,\beta^{-1}])
\]
associated to the conjugacy class $[g]$. Assume $|g| = k$, let $m = |G|/k$, and let $\Z/k \to G$ pick out $g$. Note that $(S^{\rho_G})^g$ is homeomorphic to $S^{2m}$. The Atiyah--Segal character map factors in the following way:
\[
\begin{split}
\widetilde{KU}_G^0(S^{\rho_G}) &\to \widetilde{KU}_{\Z/k}^0(S^{\rho_G}) \\ 
&\to \widetilde{KU}_{\Z/k}^0((S^{\rho_G})^g) \\
&\cong RU(\Z/k) \otimes \widetilde{KU}^0((S^{\rho_G})^g)\\
&\to \C \otimes \widetilde{KU}^0((S^{\rho_G})^g)\\
&\cong \widetilde{H}^0((S^{\rho_G})^g, \C[\beta,\beta^{-1}]),
\end{split}
\]
where the first map is induced by restriction along $\Z/k \to G$ picking out $g$, the second map is restriction along the inclusion $(S^{\rho_G})^g \to S^{\rho_G}$, the following isomorphism is due to the fact that the $\Z/k$-action on the fixed points is trivial, and the next map is induced by any map $RU(\Z/k) \to \C$ picking out a primitive $k$th root of unity.

We will trace $\beta^{\rho_G}$ through these maps. There is a commutative diagram
\[
\xymatrix{\widetilde{KU}_{\Z/k}^0(S^{\rho_G}) \ar[r]^-{\cong} & \widetilde{KU}_{\Z/k}^0(S^{m\rho_{\Z/|g|}}) \ar[r]^-{\cong} \ar[d] & \widetilde{KU}_{\Z/k}^0(S^{m\overline{\rho}_{\Z/k}}) \otimes_{RU(\Z/k)} \widetilde{KU}_{\Z/k}^0(S^{2m}) \ar[d] \\ & \widetilde{KU}_{\Z/k}^0((S^{m\rho_{\Z/k}})^{\Z/k}) \ar[r]^-{\cong} & \widetilde{KU}_{\Z/k}^0(S^{0}) \otimes_{RU(\Z/k)} \widetilde{KU}_{\Z/k}^0(S^{2m}).}
\]
We may trace $\beta^{\rho_G}$ through this diagram:
\[
\xymatrix{\beta^{\rho_G} \ar@{|->}[r] & \beta^{m\rho_{\Z/k}} \ar@{|->}[r] & \beta^{m\overline{\rho}_{\Z/k}} \otimes \beta^m \ar@{|->}[d] \\  & e(m\overline{\rho}_{\Z/k})\beta^m & e(m\overline{\rho}_{\Z/k}) \otimes \beta^m.  \ar@{|->}[l]}
\]
As restriction along $\Z/k \to G$ sends $\rho_G$ to $m\rho_{\Z/k}$, the top left isomorphism sends $\beta^{\rho_G}$ to $\beta^{m\rho_{\Z/k}}$. The right vertical mapping follows from the fact that the element $\beta^{m\overline{\rho}_{\Z/k}} \in \widetilde{KU}_{\Z/k}^0(S^{m\overline{\rho}_{\Z/k}})$ is the Thom class for $m\overline{\rho}_{\Z/k}$ and the vertical restriction map is restriction along the zero section.

Finally, the map $RU(\Z/k) \to \C$ sends the Euler class $e(m\overline{\rho}_{\Z/k})$ to $k^m$. This is because $e(m\overline{\rho}_{\Z/k}) = e(\overline{\rho}_{\Z/k})^m$ and, after fixing an isomorphism $RU(\Z/k) \cong \Z[x]/(x^k-1)$, 
we have 
\[
e(\overline{\rho}_{\Z/k}) = \prod_{i=1}^{k-1} (x^i-1). 
\] 
Setting $x = \zeta_{k}$, we get $k$ as this is the same value we get by setting $y=1$ in $(1-y^k)/(1-y)$.
\end{proof}

\section{Stable Adams operations}
\label{sec:stable}

Throughout this section, we fix an odd prime $q$ and assume that $G$ is a $q$-group. Further, let $\ell$ be a primitive root modulo $|G|$. This implies that, for any $g \in G$, the subgroup generated by $g$ is equal to the subgroup generated by $g^\ell$. Also, recall \cite{AtKthy}*{Proposition~3.2.2} that the action of the Adams operation $\psi^\ell$ on the ordinary Bott class $\beta$ is given by $\psi^\ell(\beta) = \ell \beta$.

The Adams operation $\psi^\ell \colon \widetilde{KU}_G^0(S^{\rho_G}) \to \widetilde{KU}_G^0(S^{\rho_G})$ extends to a ring endomorphism on the target of the Atiyah--Segal character map:
\[
\psi^\ell \colon \prod_{[g]} \widetilde{H}^0((S^{\rho_G})^g, \C[\beta,\beta^{-1}]) \to \prod_{[g]} \widetilde{H}^0((S^{\rho_G})^g, \C[\beta,\beta^{-1}]).
\]
An explicit formula for this map was given in \cite{BBS}*{Corollary 4.5}. Applying this formula to \cref{CharacterBottRhoG} gives
\begin{equation}
\label{FormulaPsiellCharacter}
\psi^\ell(\chi(\beta^{\rho_G}))\, \colon [g] \mapsto (|g^\ell|\ell\beta)^{|G|/|g^\ell|}.
\end{equation}

Our goal now is to compute $\psi^\ell(\beta^{\rho_G})$ under the hypotheses above. 
Compatibility of the character map with this formula for $\psi^\ell$ means that $\psi^\ell(\chi(\beta^{\rho_G})) = \chi(\psi^\ell(\beta^{\rho_G}))$. 
Since $|g| = |g^\ell|$, by \cref{CharacterBottRhoG,FormulaPsiellCharacter} we have 
\[
\chi ( \psi^\ell (\beta^{\rho_G})) = \ell^{|G|/|\bullet|} \chi(\beta^{\rho_G})
\]
where $\ell^{|G|/|\bullet|}$ sends $[g]$ to $\ell^{|G|/|g|}$.

Since the Atiyah-Segal character map $\chi$ is injective (as the $RU(G)$-modules are free), it suffices to find the finite dimensional $G$-representation with character $\ell^{|G|/|\bullet|}$.
Consider
the permutation representation
\[
\ell^{\otimes G} =  \C\{\Set(G,\underline{\ell})\}, 
\]
where $\underline{\ell}$ is a set of size $\ell$ with trivial $G$-action.
Then, for $g\in G$, the $g$-fixed points of the $G$-set $\Set(G,\underline{\ell})=\underline{\ell}^G$ are $\underline{\ell}^{G/\langle g \rangle}$.
Since the character of a permutation representation counts the cardinality of the fixed points,
$\chi(\ell^{\otimes G}) = \ell^{|G|/|\bullet|}$.
We have proved the following proposition:

\begin{prop}
Assume that $G$ is an odd $q$-group and $\ell$ is a primitive root modulo $|G|$. Let 
\[
\psi^\ell \colon \widetilde{KU}_G^0(S^{\rho_G}) \to \widetilde{KU}_G^0(S^{\rho_G})
\]
be the $\ell$th Adams operation. Then 
\[
\psi^\ell(\beta^{\rho_G}) =  \ell^{\otimes G} \beta^{\rho_G}.
\]
\end{prop}

It follows from \cite[Proposition 2.1.2]{tomDieck} that, when $\ell$ is coprime to $|G|$, $\ell^{\otimes G}$ is an invertible element in $RU(G)[\ell^{-1}]$.

\begin{prop}
Assume that $G$ is an odd $q$-group and $\ell$ is a primitive root modulo $|G|$. 
Then the Adams operation 
\[
\psi^\ell \colon RU(G)[\ell^{-1}] \to RU(G)[\ell^{-1}] 
\]
extends to a map of equivariant ring spectra
\[
\psi^\ell \colon \KUG[\ell^{-1}] \to \KUG[\ell^{-1}].
\]
\end{prop}
\begin{proof}
Since $\psi^\ell$ is a cohomology operation and $\Z \times B_GU$ represents equivariant complex $K$-theory, $\psi^{\ell}$ gives a map of $G$-spaces
\[
\psi^{\ell} \colon \Z \times B_GU \to \Z \times B_GU.
\]
Since $\psi^\ell$ is a ring map on $\pi_0$, it induces a map $\psi^\ell \colon (\Z \times B_GU)[\ell^{-1}] \to (\Z \times B_GU)[\ell^{-1}]$. To show that $\psi^\ell$ induces a map of equivariant cohomology theories, it suffices to show that $\psi^{\ell}$ can be extended to commute (up to homotopy) with the structure map for the equivariant spectrum:
\[
\xymatrix{S^{\rho_G} \wedge (\Z \times B_GU)[\ell^{-1}] \ar[r] \ar[d]_{1 \wedge \psi^\ell} & (\Z \times B_GU)[\ell^{-1}] \ar@{-->}[d]^{f} \\ S^{\rho_G} \wedge (\Z \times B_GU)[\ell^{-1}] \ar[r] & (\Z \times B_GU)[\ell^{-1}].}
\]
The structure map
\[
S^{\rho_G} \wedge (\Z \times B_GU)[\ell^{-1}] \to (\Z \times B_GU)[\ell^{-1}]
\]
is induced by the equivariant Bott map $\beta^{\rho_G} \colon S^{\rho_G} \to (\Z \times B_GU)[\ell^{-1}]$. To find $f$ such that the square commutes, it suffices to understand the two ways of going around the square on the universal map $u \colon \Z \times B_GU \to \Z \times B_GU[\ell^{-1}]$.

The two ways of going around the square give us $\beta^{\rho_G}\psi^\ell(u)$ and $f(\beta^{\rho_G}u)$. As $\ell$ has been inverted, we may set $f = \psi^\ell/\ell^{\otimes G}$, then
\[
\begin{split}
f(\beta^{\rho_G}u) &= \psi^\ell(\beta^{\rho_G}u)/\ell^{\otimes G} \\
&= \psi^\ell(\beta^{\rho_G})\psi^\ell(u)/\ell^{\otimes G} \\
&= \ell^{\otimes G}\beta^{\rho_G}\psi^\ell(u)/\ell^{\otimes G}\\
&= \beta^{\rho_G}\psi^\ell(u).
\end{split}
\]

\end{proof}

\section{The fiber of $\psi^\ell-1$}
\label{sec:fiber}

The goal of this section is to prove \cref{FiberSequenceProp,prop:maincalcp=q}, identifying the fiber of the map of equivariant spectra
\begin{equation}
\label{eq:psil1}
(\KUG)^{\wedge}_q \lra{\psi^\ell-1} (\KUG)^{\wedge}_q
\end{equation}
and identifying $\mf\pi_0$ of the fiber when $G$ is an odd $q$-group.

We begin with a lemma:

\begin{lemma}
    \label{FiberIsLocal}
If $R$ is an equivariant ring spectrum, then the $p$-completion $R^\wedge_p$ is $R/p$-local.
\end{lemma}

\begin{proof}
The usual proof goes through in the genuine equivariant setting.    Let $X$ be an $R/p$-acyclic $G$-spectrum, so that
    $X \wedge R/p \simeq X \wedge R \wedge M_G(p) \simeq \**$.
    Then $X \wedge R$ is $M_G(p)$-acyclic.
    Since $R^\wedge _p$ is $M_G(p)$-local, we have
    \[
        [X, R^\wedge_p]^G \simeq [X \wedge R, R^{\wedge}_p]_{R\text{-mod}}^G \subseteq [X \wedge R, R^\wedge_p]^G = 0.
    \]
\end{proof}

\begin{proposition}
\label{FiberSequenceProp}
For $G$ an odd $q$-group and $\ell$ a primitive root mod $|G| = q^k$ there is a fiber sequence
\begin{equation} \label{eq:fiberseq}
L_{\KUG/q}S_G \to (\KUG)^{\wedge}_q \lra{\psi^\ell-1} (\KUG)^{\wedge}_q.
\end{equation}
\end{proposition}

\begin{proof}
The canonical map of equivariant ring spectra $\eta \colon S_G \to (\KUG)^{\wedge}_q$ factors through $L_{\KUG/q}S_G$, and the induced map $L_{\KUG/q}S_G \to (\KUG)^{\wedge}_q$ is a map of rings. We wish to identify $L_{\KUG/q}S_G$ with the fiber of $\psi^\ell-1$.

     To this end, let $F_G$ denote the fiber of $\psi^\ell - 1$. Since $(\KUG)^{\wedge}_q$ is a $q$-complete equivariant commutative ring spectrum, it is $\KUG/q$-local by \cref{FiberIsLocal}. It follows that the fiber $F_G$ is $\KUG/q$-local. To identify $L_{\KUG/q}S_G$ with $F_G$, we wish to show that the canonical map $S_G \to F_G$ is an equivalence after smashing with $\KUG/q$ (the map exists because $(\psi^\ell-1)\eta=0$ and is canonical because $\mpi_1(\KUG)^{\wedge}_q$ vanishes). That is, we want the map
\[  
    \KUG/q \to \KUG/q \wedge F_G
\]  
to be an equivalence. Since the geometric fixed point functors $\{\Phi^H|H\subseteq G\}$ are jointly conservative and symmetric monoidal, it suffices to check that
\[  
\Phi^H\KUG/q \to \Phi^H \KUG/q \wedge \Phi^H F_G 
\]
is an equivalence of spectra for each $H \subseteq G$. If $H \subseteq G$ is not cyclic, then Proposition \ref{PhiNonCycKUzero} implies that $\Phi^H\KUG/q \simeq 0$. When $H \subseteq G$ is nontrivial and cyclic, Proposition \ref{CyclicGeoFP} implies that $q$ is invertible in $\Phi^H \KUG$, so again $\Phi^H\KUG/q \simeq 0$. Thus we only need to check the case $H=e$, which is the classical statement \cite{Bousfield}*{Section~4} that the $KU/q$-local sphere is the fiber of $\psi^\ell-1 \colon KU^\wedge_q \to KU^\wedge_q$. 
\end{proof}

\begin{remark}
If we complete at $p \neq q$, the strategy above does not work. For a choice of $\ell$ such that $\psi^\ell \colon (KU_G)^{\wedge}_{p} \to (KU_G)^{\wedge}_{p}$ is stable, let $F_G = \fib(\psi^\ell-1)$. By applying the geometric fixed points functor for cyclic subgroups of $G$, one can show that the canonical map
\[
(KU_G)/p \longrightarrow F_G \wedge (KU_G)/p
\]
is not an equivalence, and thus $F_G$ is not the $(KU_G)/p$-local sphere.
\end{remark}

\begin{remark}
Another approach to \cref{FiberSequenceProp} was suggested by Balderrama (see also \cite{Balderrama}). One can show that the fiber sequence is the image of the fiber sequence
\[
L_{KU/q} S \to KU^{\wedge}_q \to KU^{\wedge}_q
\]
under the functor from spectra to $G$-spectra sending $X$ to the Borel equivariant spectrum for the trivial $G$-action on $X$. One reason this works is because $(KU_G)^{\wedge}_{q}$ is Borel-complete if $G$ is a $q$-group. This follows from the fact that there is a canonical isomorphism
\[
(KU^{\wedge}_{q})^0(BG) \cong RU(G) \otimes \Z_q,
\]
for $G$ a $q$-group.
\end{remark}

We now address the algebraic analogue of \cref{FiberSequenceProp}, and give a description of the kernel of $\mpi_0(\psi^\ell-1) \colon \mf{RU} \to \mf{RU}$ in terms of the Burnside Green functor $\mf{A}$.
We will abuse notation and write $\psi^\ell-1$ for both $\mpi_0(\psi^\ell-1)$ and $\pi_0(\psi^\ell-1)$. 

Recall that linearization defines a canonical map $\mf{A} \to \mf{RU}$. This map is induced by the map sending a finite $G$-set to the associated complex permutation representation. Let 
\[
\mf{J} = \ker(\mf{A} \to \mf{RU}).
\]
Using character theory, as in \cite[Proposition 3.8]{Szymik}, it is easy to see that $\mf{J}(G/H) = J(H)$ is the ideal of $A(H)$ generated by virtual $H$-sets $[X]$ with the property that $|X^h|=0$ for $h \in H$. Thus we have a canonical {injective} map of Green functors
\[
\mf{A}/\mf{J} \hookrightarrow \mf{RU}.
\]
Note that $\mf{J}$ is also the kernel of the canonical map $\mf{A} \to \mf{R\Q}$.

\begin{proposition} \label{prop:kernel}
For $G$ an odd $q$-group and $\ell$ a primitive root mod $|G| = q^k$, we have isomorphisms of Green functors
\[
\mf{A}/\mf{J} \cong \mf{R\Q} \cong \ker(\psi^\ell-1 \colon \mf{RU} \to \mf{RU}).
\]
\end{proposition}

\begin{proof}
First, the Ritter--Segal theorem \cites{Ritter, Seg-perm}, implies that, since $G$ is a $q$-group, we have an isomorphism of Green functors $\mf{A}/\mf{J} \iso \mf{R\Q}$.
We will show that
\[
\mf{R\Q}
\cong \ker(\psi^\ell-1 \colon \mf{RU} \to \mf{RU}). 
\]
It suffices to show that we have an isomorphism of rings
\[ R\Q(G) \cong \ker( \psi^\ell-1\colon RU(G) \rtarr RU(G)).\]
The kernel
\[
\ker(\psi^\ell-1) \colon RU(G) \to RU(G)
\]
consists of the fixed points for the action of the ring endomorphism $\psi^\ell$ on $RU(G)$.

\cite{serre}*{Proposition~33} implies that 
\[
RU(G) \cong 
R\Q(\zeta_{q^k})(G).
\]
By assumption, $\ell$ is a generator of $(\Z/q^k)^\times$. For a $G$-representation $\rho$ in $R\Q(\zeta_{q^k})(G)$, \cite{tomDieck}*{Proposition 3.5.2.(i)} implies that the $\ell$th Adams operation $\psi^\ell$ acts on the character $\chi(\rho)$ through the action of 
\[
\ell \in (\Z/q^k)^{\times} \cong \Gal(\Q(\zeta_{q^k})/\Q)
\]
on the coefficients.
It follows that there is an isomorphism
\[
(R\Q(\zeta_{q^k})(G))^{\psi^\ell} \cong R\Q_{\chi}(G),
\]
where $R\Q_{\chi}(G) = \chi(RU(G)) \cap \Cl(G,\Q) \subset \Cl(G,\C)$. Now Schilling's theorem \cite{Reiner}*{Theorem 41.9} applies to $R\Q_{\chi}(G)$ since $G$ is an odd $q$-group and implies that
$
R\Q(G) \cong R\Q_{\chi}(G).
$
\end{proof}

We are now prepared to prove the following result:

\begin{prop} \label{prop:maincalcp=q}
Let $G$ be an odd $q$-group. Then there is an isomorphism of Green functors
    \[
    \mpi_0 L_{\KUG/q} S_G \cong (\mf{A}/\mf{J})^{\wedge}_q,
    \]
and $\mpi_1 L_{\KUG/q} S_G$  is finite.
\end{prop}

\begin{proof}
With the fiber sequence \eqref{eq:fiberseq} in hand, we can easily calculate $\mpi_0\left( L_{\KUG/q}S_G\right)$ and $\mpi_1\left( L_{\KUG/q}S_G\right)$. Since $\mpi_1 (\KUG)^{\wedge}_q = 0$, we have 
\[
\mpi_0 \left(L_{\KUG/q}S_G \right) \cong \ker(\psi^\ell-1) \cong (\mf{A}/\mf{J})^{\wedge}_q
\] 
by Proposition \ref{prop:kernel}. 

Now 
\[
\mpi_1 \left(L_{\KUG/q}S_G\right) \iso \coker \left( \psi^\ell - 1\colon \mpi_2 (KU_G)^\wedge_q \to \mpi_2 (KU_G)^\wedge_q
\right). 
\]
To see that $\mpi_1L_{\KUG/q}S_G$ is finite, it suffices to show that $\psi^\ell-1$ is injective on $\mpi_2 (\KUG)^\wedge_q \cong \mf{RU}^\wedge_q\{\beta\}$, where $\beta$ is the ordinary Bott class. Since $\psi^\ell-1$ is base changed along the flat extension $\Z \to \Z_q$ from the action of $\psi^\ell-1$ on $\mf{RU}\{\beta\}$, it suffices to show that the action on $\mf{RU}\{\beta\}$ is injective. Since this action is levelwise, we may show that 
\[
\psi^\ell-1 \colon RU(G)\{\beta\} \to RU(G)\{\beta\}
\]
is injective. Since $RU(G)$ is a finitely generated free $\Z$-module, we may base change to $\C$ and work with class functions, giving us 
\[
\psi^\ell-1 \colon \Cl(G)\{\beta\} \to \Cl(G)\{\beta\}.
\]
We wish to show that this map is an isomorphism. We will consider the basis consisting of indicator functions. 
The indicator functions are permuted by the action of $\psi^\ell$. If the associated permutation matrix is $S$, then $\psi^\ell$ acts on $\Cl(G)\{\beta\}$ by $\ell S$. Then we are interested in the determinant of the integer matrix $\ell S - \Id$ in which $\ell \geq 2$. Since this matrix is invertible mod $\ell$, the determinant is nonzero. 
\end{proof}

\section{The splitting of $G$-spectra, away from the order of the group} \label{sec:splitting}

For the duration of this section, we fix a prime $p$ not dividing the order of the finite group $G$. We review the fact that the $p$-local $G$-equivariant stable homotopy category splits as a product of Borel-equivariant homotopy categories. 
This essentially appeared first in \cite{GrMa}*{Appendix~A}, and also more recently in \cites{Barnes,Liu}, and explicitly in \cite{Wimmer};
as we will need an explicit description of this splitting, we prove the result in full.

This splitting arises from a corresponding splitting of the $p$-local Burnside ring of $G$ \cite{KosLocal}. The $p$-local splitting of $A(G)$ arises from the existence of certain idempotents $e_H^G \in A(G)_{(p)}$, one for each conjugacy class of subgroups. The idempotent $e^G_H$ is of the form
\begin{equation}
\label{formulaIdempotent}
    e^G_H = \frac1{|W_G(H)|}G/H + \sum_{(K)} c^H_K G/K,
\end{equation}
where $K$ runs over conjugacy classes of $G$ that are properly subconjugate to $H$ and $c^H_K \in \Z_{(p)}$.

Given the isomorphism $\pi_0(S_G) \iso A(G)$, this allows us to define, for any $p$-local $G$-spectrum $X$, the $G$-spectrum
$e^G_H X$  as the telescope 
\[ e^G_H X = \hocolim( X \xrtarr{e^G_H} X \xrtarr{e^G_H} X \xrtarr{e^G_H} \dots).
\]

When $H=\{ 1 \}$ is the trivial subgroup, this idempotent is smashing with the free $G$-space $EG_+$.
To see this consider the cofiber sequence
$EG_+ \rtarr (G/G)_+ \rtarr \widetilde{EG}$ of based $G$-spaces, which gives rise to the cofiber sequence of $G$-spectra $EG_+ \rtarr S_G \rtarr \widetilde{EG}$.
The map $e_1^G$ is the composition of maps of $G$-spectra
\begin{equation}
\label{eG1composition}
    S_G \xrtarr{\frac{\mathrm{tr}}{|G|}} G_+ \rtarr S_G.
\end{equation}
Since the underlying spectrum of $\widetilde{EG}$ is contractible, it follows that $e^G_1 \widetilde{EG} \simeq \ast$ and $e^G_1 S_G \simeq e^G_1 EG_+$. 
However, on $EG_+$, the composition \cref{eG1composition} is the identity since $EG$ has only free cells, and we conclude that $e^G_1 S_G \simeq EG_+$. We prove a generalization of this equivalence to $H \subseteq G$ in \cref{DescriptionBorelIdempotents}.

\begin{thm}[\cites{Araki,GrMa,Barnes,Liu,Wimmer}] \label{thm:primetoq}
Let $p$ be a prime not dividing the order of the group $G$.
Then the collection of geometric fixed point functors, as $(H)$ runs over conjugacy classes of subgroups, yields a symmetric monoidal equivalence of categories
\[ 
\Ho \Sp^G_{(p)} \xrtarr{(\Phi^H)}
\bigoplus_{(H)} \Ho \Sp^{hW_G(H)}_{(p)}.
\]
\end{thm}

\begin{pf}
The fact that the collection of geometric fixed point functors is fully faithful is stated as \cite[Theorem~A.16]{GrMa}, for the case of rationalization. However, the argument is based on \cite{Araki}, which provides the needed results at the level of $p$-localization, as we now recall.

We have the chain of isomorphisms
\[
\begin{split}
[X,Y]^G_p &\iso 
\bigoplus_{(H)} [ e_H^G X, e_H^G Y]^G_p 
\iso \bigoplus_{(H)} [ e^{NH}_H X, e^{NH}_H Y]^{NH}_p \\
 &\iso \bigoplus_{(H)} [ e_{1}^{WH} \Phi^H X, e_{1}^{WH} \Phi^H Y]^{WH}_p \\
&\iso \bigoplus_{(H)} [ EWH_+ \smsh \Phi^H X, EWH_+ \smsh \Phi^H Y]^{WH}_p \\
&\iso \bigoplus_{(H)} [\Phi^H X, \Phi^H Y]^{hWH}_p. \\
\end{split}
\]
Here, the second isomorphism is given by \cite[Theorem~3.5]{Araki} and  the third by \cite[Theorem~4.7]{Araki}

To see that the collection of geometric fixed point functors is essentially surjective, we provide, for each $Y \in \Sp^{hWH}_p$, a $G$-spectrum whose $K$-geometric fixed points vanish unless $K=H$, up to conjugacy, and whose $H$-geometric fixed points is $Y$. 

For a subgroup $H\leq G$, denote by $\cF[H]$ the family of subgroups of $NH$ which do not contain $H$. Since $H$ is normal in $NH$, this is indeed a family, meaning that it is closed under subgroups and conjugation. 
We then claim that the $G$-spectrum 
\[
X = \uparrow_{NH}^G \left( \widetilde{E\cF[H]} \smsh EWH_+ \smsh Y \right)
\]
has the desired fixed point properties. 

First note that the $NH$-space $\widetilde{E\cF[H]} \smsh EWH_+$ satisfies
\[
\Big( \widetilde{E\cF[H]} \smsh EWH_+ \Big)^K \simeq 
\begin{cases}
S^0 & K=H \\
\ast & \text{else.}
\end{cases}
\]
We will write $E\langle H \rangle = \widetilde{E\cF[H]} \smsh EWH_+$.

Now for any $NH$-spectrum $Z$, the double coset formula gives
\[ 
\begin{split}
\downarrow^G_{NK} \uparrow_{NH}^G Z &\simeq\bigvee \uparrow^{NK}_{NK\cap NH^g} c_g \downarrow^{NH}_{NK^{g^{-1}} \cap NH} Z \\
&\simeq\bigvee \uparrow^{NK}_{NK\cap NH^g}  \downarrow^{NH^g}_{NK \cap NH^g} c_{g^{-1}} Z.
\end{split}
\]
Then
\[\begin{split}
    \Phi^K \Big( \downarrow^G_{NK} \uparrow_{NH}^G E\langle H \rangle \smsh Y \Big) 
    &\simeq \bigvee \Phi^K \Big(\uparrow^{NK}_{NK\cap NH^g}  \downarrow^{NH^g}_{NK \cap NH^g} c_{g^{-1}} E\langle H \rangle \smsh Y \Big) \\
    &\simeq \bigvee \uparrow^{WK}_{\frac{NK\cap NH^g}K}  \Phi^K \Big( \downarrow^{NH^g}_{NK \cap NH^g} c_{g^{-1}} E\langle H \rangle \smsh Y \Big) \\
    & \simeq \begin{cases}
      Y & K = H^g \\ 
      \ast & \text{else.}
    \end{cases}
\end{split}
\]
This verifies that the collection $(\Phi^H)$ of geometric fixed point functors is essentially surjective.
Finally, the equivalence is symmetric monoidal simply because each geometric fixed point functor is symmetric monoidal.
\end{pf}

In the proof of \cref{thm:primetoq}, we employed the $p$-local idempotents $e^G_H$.  We will use the following description of the interaction of the idempotents with fixed points.

\begin{prop}
\label{DescriptionBorelIdempotents}
For $H \leq G$, $X\in \Sp^G_{(p)}$, and $p$ not dividing the order of $G$, we have 
\[
(e^G_H X)^H \simeq \Phi^H(X)
\]
in the $p$-local Borel-equivariant category $\Ho \Sp^{h W_G(H)}_{(p)}$.
\end{prop}

\begin{pf}
Since fixed points with respect to $H$ are computed by first restricting the $G$-action to the action of the normalizer $N_G(H)$, we may without loss of generality assume that $H$ is normal in $G$ and that $W_G(H) = G/H$. 

Recall that $\Phi^H(X) = (\widetilde E \mathcal{F}[H] \smsh X)^H$.
We will show that 
\[
E (G/H)_+ \smsh e^G_H X \simeq 
E (G/H)_+ \smsh \widetilde E \mathcal{F}[H] \smsh X
\]
in $\Sp^{G}_{(p)}$. 
The result then follows by passage to $H$-fixed points, since $H$ acts trivially on $E(G/H)$.
Note that, as a $G$-space, we can write $E (G/H) = E\mathcal{F}_H$, where
 $\mathcal{F}_H$ is the family of subgroups of $H$. Then 
 \[
 E \mathcal{F}[H] \times E (G/H) \simeq E \Big( \mathcal{F}[H] \cap  \mathcal{F}_H \Big) \simeq E \mathcal{P}_H, 
 \]
 where $\mathcal{P}_H$ is the family of proper subgroups of $H$.

Consider the cofiber sequence
\[
\Big( E \mathcal{F}[H] \times E (G/H) \Big)_+ \smsh X \rtarr E (G/H)_+ \smsh X \rtarr \widetilde E \mathcal{F}[H] \smsh E(G/H)_+ \smsh X.
\]
Again, the left term is $(E\mathcal{P}_H)_+\smsh X$, which is annihilated by the idempotent $e^G_H$, since all cells of $E\mathcal{P}_H$ are induced from proper subgroups of $H$. It follows that we have equivalences
\[ 
E(G/H)_+ \smsh e^G_H (X) =
e^G_H \left( E (G/H)_+ \smsh X \right) \simeq e^G_H \left( \widetilde E \mathcal{F}[H] \smsh E(G/H)_+ \smsh X \right).
\]
Since the restriction of $\widetilde E \mathcal{F}[H]$ to proper subgroups of $H$ is contractible, it follows from \cref{formulaIdempotent} that the idempotent $e^G_H$ is given on $\widetilde E \mathcal{F}[H]$ by smashing $\widetilde E \mathcal{F}[H]$ with the composition
\begin{equation}
\label{eGHcomposition}
    S_G \xrtarr{\frac{\mathrm{tr}}{[G\colon H]}} G/H_+ \rtarr S_G.
\end{equation}
On the other hand, on $E (G/H)_+$, the composition \cref{eGHcomposition} is the identity since it only has cells of type $G/H$. We conclude that 
\[
e^G_H \left( \widetilde E \mathcal{F}[H] \smsh E(G/H)_+ \smsh X \right) \simeq 
\widetilde E \mathcal{F}[H] \smsh E(G/H)_+ \smsh X. 
\]
\end{pf}

Recall that the Burnside ring $A(K)$ acts on $\mf M(G/K)$ for all $K \subseteq G$, so $A(G)$ acts on $\mf M(G/K)$ by restriction.
For $\mf M \in \Mack(G)_{(p)}$, we define
$e_H^G \mf M$ by 
$(e_H^G \mf M)(G/K) = e_H^G (\mf M(G/K))$.

The algebraic analogue of \cref{thm:primetoq}, which follows from the argument of \cite{GrMa}*{Theorem A.9 and Proposition A.12},
and \cite{BarnesKed}*{Corollary 7.3} for the monoidal structure, is as follows;

\begin{prop}
\label{MackeySplitting}
Let $p$ be a prime not dividing the order of the group $G$. Then the map
\[
\Mack(G)_{(p)} \xrightarrow{(V_H)} \bigoplus_{(H)} \Mod_{\Z_{(p)}[WH]}
\]
is a {symmetric monoidal} equivalence of categories,
where
$V_H(\underline{M}) := e^G_H \ul{M}(G/H)$.
\end{prop}

{Here the monoidal structure on $\Mod_{\Z_{(p)}[WH]}$ is given by the underlying tensor product of $\Z_{(p)}$-modules, equipped with the diagonal action of $W_G(H)$.}

\begin{remark}
    \label{rem:tambara}
    The symmetric monoidal equivalence above yields an analogous splitting of the category of Green functors localized at the prime $p$
    \[
    \mathrm{Green}(G)_{(p)} \longto \bigoplus_{(H)} \CAlg(\Mod_{\Z_{(p)}[WH]}).
    \]
    However, following the discussion after \cite{BarnesKed}*{Corollary 7.3},
    we note that this idempotent splitting does not preserve the structure of a Tambara functor.
\end{remark}

Moreover, we have the following comparison.

\begin{proposition}
\label{SplittingComparison}
The diagram
\begin{equation}
    \label{SplittingComparisonEq}
    \begin{tikzcd}
        \Ho \Sp^G_{(p)} \arrow[r,"(\Phi^H)", "\sim" swap ] \ar[d,"\ul{\pi}_n" swap] &
        \displaystyle \bigoplus_{(H)} \Ho \Sp^{hW_G(H)}_{(p)} \ar[d,"\pi_n"] \\
        \Mack(G)_{(p)} \arrow[r, "(V_H)", "\sim" swap] & 
        \displaystyle \bigoplus_{(H)} \Mod_{\Z_{(p)}[W_GH]}
    \end{tikzcd}
\end{equation}  
commutes.
\end{proposition}

\begin{proof}
Let $e^G_H \Sp^G_{(p)} \subseteq \Sp^G_{(p)}$ be the essential image of the functor $e^G_H$.  According to \cref{DescriptionBorelIdempotents}, 
    the horizontal maps factor as in the diagram
\[    \begin{tikzcd}
        \Ho \Sp^G_{(p)} \arrow[r ] \ar[d,"\ul{\pi}_n" swap] &
        \displaystyle \bigoplus_{(H)} \Ho e^G_H \Sp^G_{(p)} \ar[r,"(-)^H"] \ar[d,"\underline{\pi}_n"] &
        \displaystyle \bigoplus_{(H)} \Ho \Sp^{hW_G(H)}_{(p)} \ar[d,"\pi_n"] \\
        \Mack(G)_{(p)} \arrow[r] & 
        \displaystyle \bigoplus_{(H)} e^G_H \Mack(G)_{(p)} \arrow[r, "\mathrm{ev}_{G/H}"]  &
        \displaystyle \bigoplus_{(H)} \Mod_{\Z_{(p)}[W_GH]},
    \end{tikzcd} \] 
    where $\mathrm{ev}_{G/H}(\mf M) = \mf M(G/H)$.
    The first square commutes by construction, and
    the second square commutes by the definition of $\underline{\pi}_n$. 
\end{proof}

We will also need the following alternative description of the functor $V_H$, as suggested immediately preceding \cite{SchGlobal}*{3.4.22}.

\begin{prop}
\label{AlternativeVH}
Let $p$ be a prime not dividing the order of the group $G$. For $H~\leq~G$ and $\ul{M} \in \Mack(G)$, let 
$t_H \ul{M}\leq \ul{M}(G/H)$ be the subgroup 
generated by transfers from proper subgroups of $H$.
Assume further that $\ul{M} \in \Mack(G)_{(p)}$. 
Then the projection homomorphism
$\ul{M}(G/H) \onto e^G_H \ul{M}(G/H)=V_H(\ul{M})$
induces an isomorphism
$ \ul{M}(G/H)/t_H \ul{M} \iso V_H(\ul{M})$.
\end{prop}

\begin{pf}
The claim amounts to the statement that the kernel of the surjection $\ul{M}(G/H) \onto V_H(\ul{M})$ is precisely $t_H \ul{M}$. 
We first observe that if $K$ is (conjugate to) a proper subgroup of $H$, then the restriction $\downarrow^G_K(e^G_H) \in A(K)_{(p)}$ is 0. This implies that $e^G_H \ul{M}(G/K) = 0$ and the commuting square
\[
  \begin{tikzcd}
    \ul{M}(G/H) \arrow[r] & e^G_H \ul{M}(G/H) \\
    \ul{M}(G/K) \ar[u] \ar[r] & e^G_H \ul{M}(G/K) \ar[u]
  \end{tikzcd}
\]
shows that the image of the transfer $\ul{M}(G/K) \to \ul{M}(G/H)$ lies in the kernel. Allowing $K$ to vary over proper subgroups, we conclude that $t_H \ul{M}$ is contained in the kernel. 

On the other hand, using that $\mf M \in \Mack(G)_{(p)}$ splits as $\mf M(G/H) \cong \bigoplus_{(K)} e_K^G \mf M(G/H)$,
the kernel of the projection is a direct sum with terms $e^G_K \ul{M}(G/H)$, where $K$ is not conjugate to $H$. It remains to show that each of these lies in $t_H \ul{M}$. If $H$ is contained in $K$, up to conjugacy, then the term $e^G_K \ul{M}(G/H)$ vanishes. On the other hand, if $K$ is (conjugate to) a proper subgroup of $H$, then since $e^G_K\in A(G)_{(p)}$ is induced up from $K$, the Frobenius reciprocity axiom shows that $e^G_K \ul{M}(G/H)$ lies in the image of the transfer from the proper subgroup $K$.
\end{pf}

This description has the following consequence.

\begin{prop}
\label{prop:VHCYC}
Let $p$ be a prime not dividing the order of $G$, 
fix a $p$-local abelian group $T$ 
and let $\underline{M} \in \Mack(G)_{(p)}$. Then 
\[
    V_H(\ul{M}) \cong
    \begin{cases}
        T \  \text{with trivial $W_GH$-action} \qquad & \mbox{$H$ cyclic}\\
        0 & \mbox{else},
    \end{cases}
\]
if and only if $\ul{M} \cong \ul{A}/\ul{J} \otimes T$.
\end{prop}

\begin{proof}
Since the $p$-local marks homomorphism
\[
    A(G)_{(p)} = A(G) \otimes \mathbb Z_{(p)} \longto \prod_{(H) \leq G} \mathbb Z_{(p)}
\]
is an isomorphism \cite{tomDieck}*{Chapter 5}
and $T \cong \mathbb Z_{(p)} \otimes T$, 
we have that 
\[
    \Big(\left(\underline{A}/\underline{J} \right) \otimes T \Big) (G/K) 
    \cong 
    \prod_{(H)\leq K \text{ cyclic}} T,
\]
with restrictions and transfers the natural projections and inclusions, respectively.
\cref{AlternativeVH} then implies that
\[
    V_H\Big(\underline{A}/\underline{J} \otimes T \Big) \cong
    \begin{cases}
        T \qquad & \mbox{$H$ cyclic}\\
        0 & \mbox{else}.
    \end{cases}
\]  
Since $(V_H)$ is fully faithful by \cref{MackeySplitting}, the result follows.
\end{proof}

Combining \cref{SplittingComparison,prop:VHCYC} yields the following.
\begin{corollary}
\label{prop:ConstantPhiH}
Let $X \in \Sp^G_{(p)}$, and $T$ a fixed $p$-local abelian group. If
\[
    \pi_n \Phi^H X \cong 
    \begin{cases}
        T \  \text{with trivial $W_GH$-action} \qquad & \mbox{$H$ cyclic}\\
        0 & \mbox{else},
    \end{cases}
\]
then $\mf{\pi}_n X \cong \mf{A}/\mf{J} \otimes T$.

Further, if $X$ is a homotopy commutative equivariant ring spectrum and satisfies the condition above and $T$ is a $p$-local commutative ring, then
$\mf \pi_n X \cong \mf A / \mf J \otimes T$ as Green functors.
\end{corollary}

Here, we are using that the category of Green functors is tensored over commutative rings \cite{Lewis}*{Example 2.2(g)}.

\begin{proof}
The ``if" direction of \cref{prop:VHCYC} implies that
\[
V_H(\mf{A}/\mf{J} \otimes T) \cong \pi_n \Phi^H X,
\]
and since $V_H \mf{\pi}_n X \cong \pi_n \Phi^H X$ by \cref{SplittingComparison},
the ``only if" direction of \cref{prop:VHCYC} yields the result.

The further result follows since the equivalences in \cref{SplittingComparisonEq} are symmetric monoidal
and the $p$-local marks homomorphism is an isomorphism of rings.
\end{proof}

\begin{remark}
    \cref{prop:VHCYC,prop:ConstantPhiH} hold if we
    replace the family of cyclic subgroups of $G$ with any other family $\mathcal F$ of subgroups of $G$,
    and replace $\mf{J}$ with the Mackey ideal $\mf J_{\mathcal F}$ where $\mf{J}_{\mathcal F}(G/K)$ is generated by virtual $K$-sets such that $|X^H| = 0$ for all $H \in \mathcal F$.
\end{remark}

\section{The case $p \neq q$}
\label{sec:pneqq}

We will only need the following formal lemma in the context of \cref{thm:primetoq}; however, we state it in the highest generality the argument allows.

\begin{lemma}
    \label{lem:LocalF}
    Let $F \colon \mathcal C \to \mathcal D$ be a symmetric monoidal functor between presentable stable symmetric monoidal $\infty$-categories such that
    $\Ho(F)$ has a left adjoint $G$ that sends $F(E)$-acyclics to $E$-acyclics.
    Then, for $E, X \in \mathcal C$,
    we have $L_{F(E)}F(X) \simeq F(L_E X)$.
\end{lemma}

\begin{proof}

    Since $\mathcal C$ and $\mathcal D$ are presentable stable symmetric monoidal $\infty$-categories, there is a localization functor associated to any object.

    Next, if $Z \in \mathcal D$ is $F(E)$-acyclic,
    we have
    \[
    \Ho \mathcal D(Z, F(L_E X)) \cong \Ho \mathcal C(G(Z), L_E X) \cong \**.
    \]
    Thus $F(L_E X)$ is $F(E)$-local.
    
    It remains to check that the canonical map $F(X) \to F(L_E X)$ induces an equivalence after smashing with $F(E)$.
    We have a natural commuting diagram
    \[
    \begin{tikzcd}
    F(E) \wedge F(X) \arrow[r] \arrow[d, "\simeq"']
    &
    F(E) \wedge F(L_E X) \arrow[d, "\simeq"]
    \\
    F(E \wedge X) \arrow[r, "\simeq"'] 
    &
    F(E \wedge L_E X)
    \end{tikzcd}
    \]
    and so the top arrow is an equivalence.
\end{proof}

\begin{example}
\label{EX:EQUIV}
\cref{lem:LocalF} holds when $F = (\Phi^H)$ is the functor from \cref{thm:primetoq},
and more generally whenever a symmetric monoidal $F$ descends to an equivalence $F \colon \Ho(\mathcal C) \to \Ho(\mathcal D)$.
Let $G$ denote the inverse, 
and suppose $Z \in \mathcal D$ is $F(E)$-acyclic.
As $Z \simeq F(G(Z))$,
\[
    F(G(Z) \wedge E) \simeq F(G(Z)) \wedge F(E) \simeq Z \wedge F(E) \simeq \**.
\]  
Since $\Ho(F)$ is fully faithful, we conclude $G(Z) \wedge E \simeq \**$, i.e. $G(Z)$ is $E$-acyclic, as desired.
\end{example}

\begin{example}
\label{EX:PROJ}
\cref{lem:LocalF} holds for any projection map $F \colon \prod \mathcal C_i \to \mathcal C_j$ between symmetric monoidal $\infty$-categories.
Indeed, the right (and left) adjoint to $F$ is given by 
\[
    G(X_j)_i = \begin{cases}
        X_j \qquad & i=j\\
        \** & \mbox{otherwise.}
    \end{cases} 
\]
Then, for any $E = (E_i) \in \prod \mathcal C_i$, it is clear that if $Z_j$ is $F((E_i))=E_j$-acyclic then $G(Z_j)$ is $E$-acyclic.
\end{example}

\begin{example}
\label{EX:U}
\cref{lem:LocalF} holds for the forgetful map $u\colon \Sp^{hG} \to \Sp$ from Borel $G$-spectra to underlying spectra.
The left adjoint sends $X$ to $EG_+ \wedge \mathrm{inf}_{G/G}^G X$, 
and $u(E) \wedge Z \simeq \**$ in $\Sp$ iff $E \wedge(EG_+ \wedge \mathrm{inf}_{G/G}^G Z) \simeq \**$ in $\Sp^{hG}$. 
\end{example}

The $p \neq q$ analogue of \cref{prop:maincalcp=q}
now follows from the calculation of $\pi_n L_{KU/p}S$ from \cite[Corollary 4.5]{Bousfield} and the following stronger result:
\begin{proposition}
    \label{prop:pneqq}
    Let $G$ be an odd $q$-group, and $p\neq q$. Then 
    we have an isomorphism of graded Green functors
    \begin{equation}
        \label{eq:pneqq}
        \mf{\pi}_* L_{KU_G/p}S_G \cong \mf{A}/\mf{J} \otimes \pi_* L_{KU/p}S
    \end{equation}
    and moreover an equivalence of (homotopy) commutative equivariant ring spectra
    \begin{equation}
        \label{eq:ECYCL}
    L_{KU_G/p} S_G \simeq ECyc_+ \wedge \mathrm{inf}_{G/G}^G L_{KU/p}S,
    \end{equation}
    where $ECyc$ is the universal space for the family of cyclic subgroups of $G$.
\end{proposition}

\begin{proof}
\cref{lem:LocalF} implies that
applying \cref{EX:EQUIV,,EX:PROJ,,EX:U} to the composite
\[
\Ho \Sp^G_{(p)} 
\xrtarr{(\Phi^H)}
\bigoplus_{(H)} \Ho \Sp^{hW_G(H)}_{(p)}
\longto
\Ho \Sp^{hW_G(H)}_{(p)} 
\longto
\Ho \Sp_{(p)},
\]
yields an equivalence
\[
\Phi^H L_{(KU_G)/p} S_G  \simeq L_{\Phi^H (KU_G)/p} \Phi^H S_G
\]
as non-equivariant spectra.
Since $\Phi^H$ preserves cofiber sequences, \cref{CyclicGeoFP,PhiNonCycKUzero} imply that
$\Phi^H ((KU_G)/p) \simeq \Phi^H(KU_G)/p$ is a free $KU/p$-module for $H$ cyclic, and contractible otherwise.
Thus, as non-equivariant spectra,
\begin{equation}
\label{eq:PHIHSKUG}
\Phi^H L_{(KU_G)/p} S_G  \simeq 
\begin{cases}
L_{\bigvee KU/p} S \simeq  L_{KU/p} S \qquad & \mbox{$H$ cyclic}\\
\** & \mbox{otherwise.}
\end{cases} 
\end{equation}
\cref{prop:ConstantPhiH} then implies \cref{eq:pneqq}.

Towards \cref{eq:ECYCL}, we note that
\[
    \Phi^H(ECyc_+ \wedge \mathrm{inf}^G_{G/G} L_{KU/p}S) \simeq 
    (ECyc^H)_+ \wedge \Phi^H(\mathrm{inf}^G_{G/G} L_{KU/p}S) \simeq
    \begin{cases}
        L_{KU/p}S \qquad & \mbox{$H$ cyclic}\\
        \** & \mbox{else},
    \end{cases}
\]
so that the geometric fixed points of $L_{KU_G/p}S_G$ agree with those of $ECyc_+ \wedge \mathrm{inf}^G_{G/G} L_{KU/p}S$.
It remains only to produce a map of $E_\infty$-rings of $G$-spectra 
\[ \mathrm{inf}_{G/G}^G L_{KU/p}S \rtarr L_{KU_G/p}S_G,\]
or equivalently a map of $E_\infty$-rings
\[
L_{KU/p}S \rtarr (L_{KU_G/p}S_G)^G.
\]
In other words, it suffices to show that $(L_{KU_G/p}S_G)^G$ is $KU/p$-local.

Thus let $X$ be a $KU/p$-acyclic. We wish to show that $[ X, (L_{KU_G/p}S_G)^G]=0$. The assumption is equivalent to the statement that $X/p$ is $KU$-acyclic. The vanishing is equivalent to the vanishing of 
\[
[ \mathrm{inf}_{G/G}^G X, L_{KU_G/p}S_G ]^G.
\]
Thus it suffices to show that $\mathrm{inf}_{G/G}^G X$ is $KU_G/p$-acyclic, or equivalently that $\mathrm{inf}_{G/G}^G X/p$ is $KU_G$-acyclic. But
\[
\mathrm{inf}_{G/G}^G X/p \wedge \mathrm{inf}_{G/G}^G KU \simeq \mathrm{inf}_{G/G}^G (X/p \wedge KU) \simeq \ast,
\]
so the result follows by base change along the $E_\infty$-ring map $\mathrm{inf}_{G/G}^G KU \rtarr KU_G$.
\end{proof}

\section{Computing using the arithmetic fracture square}
\label{sec:fracturesquare}

We need one final lemma.
\begin{lemma}
\label{lem:Q}
For all finite groups $G$ there is an equivalence of rational equivariant spectra
\[
\mathbb Q \otimes L_{KU_G} S_G \simeq H(\Q \otimes \underline{A}/\underline{J}),
\]
where $H(\Q \otimes \underline{A}/\underline{J})$ is the equivariant Eilenberg-MacLane spectrum.
\end{lemma}

\begin{proof}
It follows from \cite[Proposition 2.11]{Bousfield}, that $\Q \otimes L_{KU_G} S_G \simeq L_{\Q \otimes KU_G} S_G$. The functor $(\Phi^H)$ is an equivalence from the category of rational $G$-equivariant spectra to the product of the rational Borel-equivariant categories.

Now, by \cref{lem:LocalF}, we have equivalences
$\Phi^H(L_{\Q \otimes KU_G} S_G) \simeq L_{\Phi^H(\Q \otimes KU_G)} \Phi^HS_G \simeq L_{\Phi^H(\Q \otimes KU_G)} S$
of non-equivariant spectra.
Moreover, we have
\[
     L_{\Phi^H(\Q \otimes KU_G)} S \simeq
    \begin{cases}
    L_{H\mathbb Q}S \simeq H \mathbb Q \qquad & \mbox{$H$ cyclic}\\
    \** \qquad & \mbox{else},
    \end{cases}
\]
since $\Q \otimes \Phi^H KU_G \simeq *$ when $H$ is not cyclic,
and 
$\Q \otimes \Phi^H KU_G$ is a nontrivial $H\Q$-module when $H$ is cyclic.
The rational analogue of \cref{prop:ConstantPhiH} then implies that 
$\mathbb Q \otimes L_{KU_G} S_G \simeq H(\Q \otimes \underline{A}/\underline{J})$.
\end{proof}

We can now prove our main result.

\begin{proof}[Proof of \cref{thm:main}]
Adapting \cref{eq:fracsq} with $X = L_{KU_G}S_G$ yields the following homotopy pullback square of 
$E_\infty$-rings in $G$-spectra
\begin{equation}
\label{FractureSquare}
    \begin{tikzcd}
        L_{KU_G} S_G \arrow[d] \arrow[r] 
        &
        \displaystyle\prod_{p} L_{KU_G/p} S_G \arrow[d,"g"]
        \\
        \mathbb Q \otimes L_{KU_G} S_G \arrow[r,"f"]
        &
        \mathbb Q \otimes \displaystyle\prod_{p} L_{KU_G/p} S_G.
    \end{tikzcd}    
\end{equation}
This yields a long-exact sequence in Mackey functor-valued homotopy. Since $\underline{\pi}_1 L_{KU_G/p}S_G$ is trivial except when $p=2$ or $p=q$, when it is torsion by \cite{Bousfield}*{Corollary~4.5} with \cref{prop:pneqq} or \cref{prop:maincalcp=q}, we have the long-exact sequence:
\[
    0 \longto
    \underline{\pi}_0 L_{KU_G}S_G \longto
    \left( \mathbb Q \otimes \underline{\pi}_0 L_{KU_G} S_G \right) \oplus \left( \prod \underline{\pi}_0 L_{KU_G/p} S_G \right) \overset{f-g}{\longto}
    \mathbb Q \otimes \prod \underline{\pi}_0 L_{KU_G/p} S_G.
\]
Applying \cref{lem:Q,,prop:maincalcp=q,,prop:pneqq}, this is the long exact sequence of Mackey functors
\[
0 \longto 
\underline{\pi}_0 L_{KU_G}S_G \longto 
(\Q \otimes \underline{A}/\underline{J}) \oplus (\prod_{p} (\mf{A}/\mf{J})^{\wedge}_p) \times (\mf{A}/\mf{J} \otimes \mathbb F_2)
\xrightarrow{f-g}
\Q \otimes \left( \prod_{p} (\mf{A}/\mf{J})^{\wedge}_p \times (\mf{A}/\mf{J} \otimes \mathbb F_2) \right).
\]
The factor containing $\F_2$ arises from the fact that $\pi_0L_{KU/2}S \cong \Z_2 \oplus \F_2$. Note that there is an isomorphism of Mackey functors 
\[
\Q \otimes \left( \prod_{p} (\mf{A}/\mf{J})^\wedge_p \times (\mf{A}/\mf{J} \otimes \mathbb F_2) \right)
\cong \Q \otimes \left( \prod_{p} (\mf{A}/\mf{J})^{\wedge}_p \right). 
\]
It follows that $\mf{A}/\mf{J} \otimes \mathbb F_2$ is in the kernel of $f-g$. 
The remaining part of the exact sequence is the arithmetic fracture square for $\underline{A}/\underline{J}$. As $A(H)/J(H)$ is a finitely generated free abelian group for each $H \subseteq G$, we have $\underline{A}/\underline{J} \oplus (\mf{A}/\mf{J} \otimes \mathbb F_2) = \ker(f-g)$. 
\end{proof}

Given \cref{thm:main}, it is natural to wonder if there is an isomorphism 
\[
\ul{\pi}_i L_{KU_G} S_G \cong \ul{A}/\ul{J} \otimes \pi_i L_{KU} S
\]
or even if $(L_{KU_G}S_G)^H \simeq L_{KU}S \otimes A(H)/J(H)$. This is already false for $i=1$.

\begin{prop}
For $G = C_3$, we have 
\[
\ul{\pi}_1 L_{KU_G} S_G \cong \big( \ul{A} \otimes \pi_1 L_{KU}S \big) \oplus \ul{T},
\]
where $\ul{T}$ is the unique $C_3$-Mackey functor with $\ul{T}(G/e) = 0$ and $\ul{T}(G/G) = \Z/3$. 
\end{prop}

\begin{pf}
We first recall \cite{Bousfield}*{Corollary~4.5, Corollary~4.6} that $\pi_1 L_{KU} S\iso \pi_1 L_{KU/2}S \iso (\Z/2)^2$ and that $\pi_1 L_{KU/p} S \iso 0$ if $p\neq 2$.
The pullback square \cref{FractureSquare} implies that \[ \mpi_1 \left( L_{KU_G}S_G \right) \iso \mpi_1 \Big( \displaystyle\prod_{p} L_{KU_G/p} S_G \Big) \iso \prod_p \mpi_1 \left( L_{KU_G/p} S_G \right). 
\]
\cref{prop:pneqq} gives that, for $p\neq 3$, 
\[ 
\mpi_1 L_{KU_G/p} S_G  \iso \mf{A} \otimes \pi_1 L_{KU/p} S,
\]
since the ideal $\ul{J}$ vanishes if $G$ is cyclic. 
Thus it follows that 
\[
\ul{\pi}_1 L_{KU_G} S_G \iso \mpi_1\left( L_{KU_G/3} S_G \right) \oplus \left(\ul{A} \otimes \pi_1 L_{KU/2}S\right).
\]
It remains to determine $\mpi_1\left( L_{KU_G/3} S_G \right)$.

We may compute $\mpi_1 \left( L_{KU_G/3} S_G\right) $
by use of the fiber sequence \cref{eq:fiberseq}.
As in \cref{prop:maincalcp=q}, this Mackey functor can be computed as
\[
\mpi_1 \left(L_{\KUG/3}S_G\right) \iso \coker \left( \psi^\ell - 1\colon (\ul{RU}_G)^\wedge_3\{\beta\} \to (\ul{RU}_G)^\wedge_3\{\beta\}
\right). 
\]
Here, we may take $\ell=2$.
At the underlying level, this is the classical $\pi_1(L_{KU/3}S)$, which vanishes. At the fixed point level, let us again write $RU(C_3) = \Z[x]/(x^3-1)$. Then $\psi^2(x\cdot \beta) = 2 x^2 \cdot \beta$ and similarly $\psi^2(x^2\cdot \beta) = 2x\cdot \beta$. The homomorphism
\[
\psi^2-1\colon RU(C_3)^\wedge_3 \{\beta\} \rtarr RU(C_3)^\wedge_3 \{\beta\}
\]
may therefore be represented by the matrix
\[
\begin{pmatrix}
2-1 & 0 & 0 \\
0 & -1 & 2 \\
0 & 2 & -1
\end{pmatrix}
\sim
\begin{pmatrix}
1 & 0 & 0 \\
0 & -1 & 2 \\
0 & 0 & 3
\end{pmatrix}.
\]
We conclude that the cokernel is isomorphic to $\Z/3$.
\end{pf}

\section{The $G$-$E_\infty$-ring structure on $L_{\KUG}S_G$}
\label{sec:Hillforthewin}

In this final section, we apply the results of \cref{sec:geometric} and \cite{Hill19} to show that $L_{\KUG}S_G$ is a $G$-$E_\infty$-ring spectrum when $G$ is an odd $q$-group. This implies that $\mpi_0L_{\KUG}S_G$ is a Tambara functor.
Moreover, we determine this structure in the case where $G$ is cyclic.

We will make use of the norm construction
$ N_H^G \colon Sp^H \rtarr Sp^G$
(see \cite{HHR}*{Section~2.2.3}).
This lifts to a functor on $H$-$E_\infty$-rings, where it participates in an adjunction
\begin{equation}
\label{NormResAdjunction}
\begin{tikzcd}
H\mathrm{-}E_{\infty}\mathrm{-ring}(Sp^H) \arrow[r, shift left = 1ex, "N_H^G"]
&
G\mathrm{-}E_{\infty}\mathrm{-ring}(Sp^G). \arrow[l, shift left = 1ex, "\downarrow^G_H"]
\end{tikzcd}
\end{equation}
We will also follow \cite{Hill19} in writing $N^{G/H}$ for the composite functor $N_H^G \circ \downarrow^G_H$ on $G$-spectra. More generally, by decomposing a finite $G$-set $T$ into a disjoint union of orbits, the norm $N^T$ can be interpreted as the smash product of norms of the form $N^{G/H_i}$, as in \cite{Hill19}*{Definition~2.2}.

\begin{proposition}
\label{LKUGSH_GEinf}
For $G$ an odd $q$-group, $L_{\KUG}S_G$ admits the structure of a $G$-$E_{\infty}$-ring.
\end{proposition}
\begin{proof}
By \cite[Theorem 3.9]{Hill19}, it suffices to show that for $L \subseteq G$ and each $L$-set $T$, the norm $N^{G \times_L T}(-)$ preserves $\KUG$-acyclics. After decomposing $T$ into transitive $L$-sets $\coprod_i L/H_i$, we have
\[
N^{G \times_L T}(X) \cong \bigwedge_i N^{G/H_i}(X).
\]
Thus it suffices to show that each norm $N^{G/H_i}(-)$ preserves $\KUG$-acyclics. Therefore we may assume that $G \times_L T \cong G/H$.

Let $X$ be a $\KUG$-acyclic $G$-spectrum. Since $\Phi^K \KUG \simeq \ast$ for $K \subseteq G$ noncyclic, this is equivalent to the statement that $\Phi^K X$ is $\Phi^K \KUG$-acyclic for $K \subseteq G$ cyclic. Further, since $\Phi^K \KUG$ is a free $KU[\frac{1}{q}]$-modules for $K$ nontrivial cyclic and $\Phi^e \KUG$ is a free $KU$-module, we have that $\Phi^K X \wedge KU[\frac{1}{q}] \simeq \ast$ for $K$ nontrivial  cyclic and $\Phi^e X \wedge KU \simeq \ast$. 

Thus, to show that $N^{G/H}X$ is $\KUG$-acyclic, it suffices by \cite{Hill19}*{Proposition 3.2} to show that 
$\Phi^K(N^{G/H}X)$ is $KU[\frac{1}{q}]$-acyclic for $K$ nontrivial cyclic and that $\Phi^e(N^{G/H}X)$ is $KU$-acyclic. \cite[Lemma 2.1]{Hill19} provides us with an equivalence
\[
\Phi^K(N^{G/H}X) \simeq \bigwedge_{KgH \in K\\ G/H} \Phi^{K^g \cap H} X,
\]
where $K^g = gKg^{-1}$. For $K$ a nontrivial cyclic subgroup of $G$, $K^g \cap H$ is cyclic and may be trivial. Either way, it follows that $\Phi^{K^g \cap H} X \wedge KU[\frac{1}{q}] \simeq \ast$. If $K$ is trivial, then all of the factors are $\Phi^eX$ and these are $KU$-acyclic. Thus $\Phi^K(N^{G/H}X)$ is $\Phi^K \KUG$-acyclic for each cyclic subgroup $K \subseteq G$ and $N^{G/H} X$ is $\KUG$-acyclic.
\end{proof}

Recall from \cite{HHR}*{Section~2.3.3} that if $X$ is a $G$-$E_\infty$-ring, then 
for $x \in \mf{\pi}_0^H(X)$, the norm on $x$ may be calculated as the composition
\[ 
S_G \simeq N_H^G(S_H) \xrtarr{N_H^G(x)} N_H^G(\downarrow^G_H X) \xrightarrow{\varepsilon} X,
\]
where $\varepsilon$ is the counit of the adjunction \eqref{NormResAdjunction}.

In the case where $G = C_{q^k}$ is a cyclic $q$-group, we will simplify our notation and write
\[
    N_i^j = N_{C_{q^i}}^{C_{q^j}} \qquad \text{and} \qquad R_i^j = R_{C_{q^i}}^{C_{q^j}}
\]
for the norm and restriction maps of a $C_{q^k}$-Tambara functor.

The following lemma was suggested to us by Balderrama. 
\begin{lemma}[Balderrama]
    \label{lemma:Balderrama}
    Let $G = C_{q^k}$ be a cyclic odd $q$-group. For $0 \leq i \leq k$, 
    let $x_i$ be the generator of 
    \begin{equation}
        \label{eq:LKUSCyc}
        {\pi}_0^{C_{q^i}} L_{KU_{C_{q^k}}}S_{C_{q^k}} \cong 
        A(C_{q^i})[x_i]/(x_i^2,2x_i).
    \end{equation}
    Then for all $0 \leq i \leq j \leq k$, $N_{i}^{j}(x_i) \neq 0$.
\end{lemma}

\begin{proof}
As $x_0^2 = 0 $, it follows that $N_{0}^{i}(x_0)^2=0.$
Since $A(G)$ has no nilpotents, we conclude that $x_i$ divides  $N_{0}^{i}(x_0)$. 
Thus, it suffices to prove that $N_{0}^{k}(x_0) \neq 0$. 
By the discussion above, we see that $N_{0}^{k}(x_0)$ is the composite
\begin{equation}
    \label{eq:NX}
    S_{C_{q^k}} \xrtarr{N_{e}^{C_{q^k}}(x_0)} 
    N_{e}^{C_{q^k}}(\downarrow^{C_{q^k}}_e L_{KU_{C_{q^k}}}S_{C_{q^k}}) \xrightarrow{\varepsilon} L_{KU_{C_{q^k}}}S_{C_{q^k}}.
\end{equation} 
Applying the geometric fixed point functor $\Phi^{C_{q^k}}(-)$ to this gives the composite
\[
S \xrightarrow{x_0} L_{KU}S \to L_{KU}S \left[\frac{1}{q}\right]
\]
by \cite{HHR}*{Section~2.5.4} and \cref{CyclicGeoFP}.
Since $x_0$ is not $q$-torsion (since $q$ is odd), this map is nonzero. Thus the original composite \cref{eq:NX} must be nonzero as well.
\end{proof}

We recall that $A(C_{q^i})$ is a free abelian group with generators $y_j = [C_{q^i}/C_{q^{j}}]$ for $0 \leq j \leq i$. Then $y_i=1$ is the unit for the ring structure, and $A(C_{q^i})$ is generated as a ring by the $y_j$ with $j\neq i$.
The restriction map $R^{i+1}_i \colon A(C_{q^{i+1}}) \to A(C_{q^i})$ 
is given on the multiplicative generators by $R^{i+1}_i(y_j)=q y_j$, where $0 \leq j < i+1$.

Since the Tambara functor structure on $\mf{A}$ is known, the following proposition determines the Tambara functor structure on $\mf{\pi}_0 L_{KU_{C_{q^k}}}S_{C_{q^k}}$.

\begin{proposition}
    Let $G = C_{q^k}$ be a cyclic odd $q$-group.
    With notation as in \cref{eq:LKUSCyc} and above, $N_i^{i+1}(x_i) = x_{i+1}(1+y_i)$. 
\end{proposition}    

\begin{proof}
    Since $x_i^2 = 0$ and $A(G)$ has no nilpotents, we know that
    \[
    N_i^{i+1}(x_i) = x_{i+1}\left( a_{i+1} + a_i y_i + a_{i-1}y_{i-1} \dots + a_0 y_0 \right).
    \]
    Here, the coefficients $a_j$ can be taken to be 0 or 1, as $2 x_{i+1}=0$.
    We then have
    \begin{align*}
        R^{i+1}_i N_i^{i+1}(x_i) 
        &= x_i \left( a_{i+1} + a_{i}q + a_{i-1}q y_{i-1} + \dots + a_0 q y_0 \right)\\
        &= x_i \left(a_{i+1} + a_i + a_{i-1}y_{i-1} + \dots + a_0 y_0 \right),
    \end{align*}
    where the last equality follows from the fact that $q$ is odd. 
    Since $R^{i+1}_i N_i^{i+1}(x_i) = x_i^q = 0$, 
    it follows that $a_{i-1} = a_{i-2} = \dots = a_0 = 0$ and $a_{i+1} + a_i \in 2 \mathbb Z$.
    But by \cref{lemma:Balderrama}, $N_i^{i+1}(x_i) \neq 0$, and so we must have $a_{i+1} = a_i = 1$.
    Thus
    $N_i^{i+1}(x_i) = x_{i+1}(1+y_i)$.
\end{proof}


\end{document}